\documentclass[11pt,reqno]{amsart}

\usepackage[margin=1.5in]{geometry}
\usepackage[colorinlistoftodos]{todonotes}

\newcounter{todocounter}

\usepackage[T1]{fontenc}
\usepackage{tgpagella}
\usepackage{mathpazo}

\usepackage[pagebackref]{hyperref} 
\usepackage{xcolor}
\usepackage{amsmath,amsthm}
\usepackage{amssymb}
\usepackage{graphicx}
\usepackage[mathscr]{eucal}
\usepackage{MnSymbol}
\usepackage[frame,ps,matrix,arrow,curve,rotate,all,2cell,tips]{xy}
\usepackage{epic,eepic}
\usepackage{enumerate}

\newtheorem{theorem}[equation]{Theorem}
\newtheorem{lemma}[equation]{Lemma}

\newtheorem{proposition}[equation]{Proposition}
\newtheorem{corollary}[equation]{Corollary}

\theoremstyle{definition}
\newtheorem{definition}[equation]{Definition}
\newtheorem{example}[equation]{Example}
\newtheorem{remark}[equation]{Remark}
\newtheorem{assumption}[equation]{Assumption}
\newtheorem{convention}[equation]{Convention}
\newtheorem{notation}[equation]{Notation}

\numberwithin{equation}{subsection}

\newcommand{\M}{\mathcal{M}}

\newcommand{\mtosdash}{\M^{S(-)}}
\newcommand{\mtosofc}{\M^{\sofc}}

\newcommand{\Msigman}{\M^{\Sigma_n}}
\newcommand{\Msigmanop}{\M^{\sigmaop_n}}

\newcommand{\N}{\mathcal{N}}

\newcommand{\ntosdash}{\N^{S(-)}}

\renewcommand{\emptyset}{\varnothing}

\newcommand{\Ho}{\mathsf{Ho}}

\newcommand{\po}{\ar@{}[dr]|(.7){\Searrow}}
\newcommand{\pb}{\ar@{}[dr]|(.3){\Nwarrow}}

\newcommand{\ch}{\mathsf{ch}}
\newcommand{\Ch}{\mathsf{Ch}}
\newcommand{\Chk}{\Ch(k)}
\newcommand{\chk}{\ch(k)}

\newcommand{\boxprod}{\mathbin\square}

\newcommand{\tensorover}[1]{\underset{#1}{\otimes}}
\newcommand{\tensoroversigman}{\underset{\Sigma_n}{\otimes}}

\newcommand{\algo}{\alg(\sO)}
\newcommand{\algosubc}{\alg(\osubc)}
\newcommand{\algpsubc}{\alg(\psubc)}
\newcommand{\algrosubc}{\alg(R\osubc)}
\newcommand{\algso}{\alg(S,\sO)}
\newcommand{\algsp}{\alg(S,\sP)}
\newcommand{\algsoone}{\alg(S^1,\sO^1)}
\newcommand{\algsotwo}{\alg(S^2,\sO^2)}
\newcommand{\algass}{\alg(\Ass)}
\newcommand{\algassc}{\alg(\Assc)}
\newcommand{\algsass}{\alg(S,\Ass)}
\newcommand{\algsassone}{\alg(S^1,\Assone)}
\newcommand{\algsasstwo}{\alg(S^2,\Asstwo)}

\newcommand{\algp}{\alg(\sP)}
\newcommand{\algq}{\alg(\sQ)}

\usepackage{tikz}
\usetikzlibrary{arrows,decorations.pathmorphing}
\usetikzlibrary{backgrounds,positioning}
\usetikzlibrary{fit,petri,shapes.misc}

\tikzset{auto}

\tikzset{empty/.style={circle,inner sep=0pt,minimum size=6mm}}
\tikzset{emptyvt/.style={circle,inner sep=0pt,minimum size=0mm}}

\tikzset{plain/.style={circle,draw,very thick,
inner sep=0pt,minimum size=6mm}}

\tikzset{fatplain/.style={rounded rectangle,draw,very thick,minimum size=6mm}}
\tikzset{bigplain/.style={rounded rectangle,draw,very thick,minimum size=.8cm}}

\tikzset{yellowvt/.style={circle,draw,fill=yellow,very thick,inner sep=0pt,minimum size=6mm}}
\tikzset{bluevt/.style={circle,draw,fill=blue!20,very thick,inner sep=0pt,minimum size=6mm}}
\tikzset{greenvt/.style={circle,draw,fill=green!30,very thick,inner sep=0pt,minimum size=6mm}}
\tikzset{redvt/.style={circle,draw,fill=red!30,very thick,inner sep=0pt,minimum size=6mm}}

\tikzset{arrow/.style={->,thick}}
\tikzset{dashedarrow/.style={->,dashed,thick}}
\tikzset{dottedarrow/.style={->,dotted,thick}}
\tikzset{mapto/.style={|->,thick}}

\tikzset{implies/.style={thick,double,double equal sign distance,-implies}} 

\tikzset{line/.style={thick}}
\tikzset{dottedline/.style={dotted,thick}}
\tikzset{dashedline/.style={dashed,thick}}

\tikzset{inputleg/.style={<-,thick}}
\tikzset{outputleg/.style={->,thick}}
\tikzset{dottedinput/.style={<-,dotted,thick}}

\newcommand{\adjoint}{\hspace{-.1cm}
\nicearrow\xymatrix{ \ar@<2pt>[r] & \ar@<2pt>[l]}\hspace{-.1cm}}
\renewcommand{\hookrightarrow}{\nicexy{\ar@{^{(}->}[r] &}}
\newcommand{\nicearrow}{\SelectTips{cm}{10}}
\newcommand{\nicexy}{\nicearrow\xymatrix@C+10pt@R+10pt}
\newcommand{\narrowxy}{\nicearrow\xymatrix@R+10pt}

\newcommand{\pushout}{\ar@{}[dr]|(0.75){\Searrow}}
\newcommand{\drrpushout}{\ar@{}[drr]|(0.90){\Searrow}}

\newcommand{\comp}{\circ}
\newcommand{\defn}{\,\overset{\mathrm{def}}{=\joinrel=}\,}
\newcommand{\fbar}{\overline{f}}

\newcommand{\pizero}{\pi_0}
\newcommand{\pizerom}{\pizero^{\M}}
\newcommand{\pizeron}{\pizero^{\N}}
\newcommand{\Rbar}{\overline{R}}
\renewcommand{\to}{\hspace{-.1cm}\nicearrow\xymatrix@C-.2cm{\ar[r]&}\hspace{-.1cm}}

\newcommand{\bbI}{\mathbb{I}}
\newcommand{\bbN}{\mathbb{N}}
\newcommand{\bbNsq}{\bbN^{\times 2}}

\newcommand{\tensorunit}{\mathbb{1}}
\newcommand{\oneone}{\tensorunit \amalg \tensorunit}
\newcommand{\tensorunitm}{\tensorunit^{\M}}
\newcommand{\tensorunitn}{\tensorunit^{\N}}

\newcommand{\frakC}{\mathfrak{C}}
\newcommand{\fC}{\mathfrak{C}}

\newcommand{\csquare}{\fC^{\times 2}}

\newcommand{\alphac}{\alpha_{\fC}}
\newcommand{\alphad}{\alpha_{\fD}}
\newcommand{\alphacbar}{\overline{\alphac}}
\newcommand{\alphabarc}{\alphacbar}
\newcommand{\alphastar}{\alpha^{\ast}}
\newcommand{\alphaone}{\alpha^{1}}
\newcommand{\alphatwo}{\alpha^{2}}
\newcommand{\Ralpha}{R_{\alpha}}
\newcommand{\Ralphac}{\Ralpha^{\fC}}

\newcommand{\betastar}{\beta^{\ast}}
\newcommand{\betastaro}{\betastar_{\sO}}

\newcommand{\epsilonc}{\epsilon_{\fC}}
\newcommand{\epsilonstar}{\epsilon^{\ast}}
\newcommand{\epsilonstaro}{\epsilonstar_{\sO}}

\newcommand{\etac}{\eta_{\fC}}
\newcommand{\etacstar}{\etac^{\ast}}
\newcommand{\iotac}{\iota_{\fC}}

\newcommand{\fD}{\mathfrak{D}}

\newcommand{\cat}{\mathsf{cat}}
\newcommand{\Acat}{A^{\cat}}
\newcommand{\Bcat}{B^{\cat}}
\newcommand{\Ccat}{C^{\cat}}
\newcommand{\Dcat}{D^{\cat}}
\newcommand{\fcat}{f^{\cat}}
\newcommand{\gcat}{g^{\cat}}
\newcommand{\hcat}{h^{\cat}}

\newcommand{\sG}{\mathsf{G}}
\newcommand{\sGc}{\sG_{\fC}}
\newcommand{\sGd}{\sG_{\fD}}

\newcommand{\Gm}{\sG^{\M}}
\newcommand{\Gmc}{\Gm_{\fC}}
\newcommand{\Gn}{\sG^{\N}}
\newcommand{\Gnc}{\Gn_{\fC}}

\newcommand{\sI}{\mathsf{I}}

\newcommand{\sO}{\mathsf{O}}
\newcommand{\sOop}{\sO^{\mathrm{op}}}
\newcommand{\Obar}{\overline{\sO}}

\newcommand{\osubc}{\sO_{\fC}}
\newcommand{\osubcone}{\sO^1_{\fC}}
\newcommand{\osubctwo}{\sO^2_{\fC}}
\newcommand{\osubd}{\sO_{\fD}}

\newcommand{\sP}{\mathsf{P}}
\newcommand{\psubc}{\sP_{\fC}}
\newcommand{\psubd}{\sP_{\fD}}

\newcommand{\sQ}{\mathsf{Q}}

\newcommand{\As}{\mathsf{As}}
\newcommand{\Asc}{\As_{\fC}}
\newcommand{\Ascs}{\Asc^S}
\newcommand{\Asd}{\As_{\fD}}
\newcommand{\Asds}{\Asd^S}
\newcommand{\Ass}{\As^S}
\newcommand{\Assone}{\As^{S^1}}
\newcommand{\Asstwo}{\As^{S^2}}
\newcommand{\Assc}{\Ass_{\fC}}

\newcommand{\tg}{\mathtt{G}}
\newcommand{\tgc}{\tg_{\fC}}

\newcommand{\tone}{\mathtt{1}}

\newcommand{\Lbar}{\overline{L}}
\newcommand{\Lbarc}{\Lbar^{\fC}}

\newcommand{\alphabar}{\overline{\alpha}}
\newcommand{\betabar}{\overline{\beta}}

\newcommand{\varphibar}{\overline{\varphi}}

\newcommand{\ua}{\underline{a}}
\newcommand{\ub}{\underline{b}}
\newcommand{\uc}{\underline{c}}

\newcommand{\ud}{\underline{d}}

\newcommand{\smallop}{{\scalebox{.5}{$\mathrm{op}$}}}

\newcommand{\calc}{\mathcal{C}}
\newcommand{\cald}{\mathcal{D}}
\newcommand{\caldop}{\mathcal{D}^{\smallop}}

\newcommand{\G}{\mathcal G}
\newcommand{\Gc}{\G_{\fC}}

\newcommand{\cJ}{\mathcal{J}}

\newcommand{\calm}{\mathcal{M}}
\newcommand{\calmc}{\calm^{\fC}}

\newcommand{\Cat}{\mathsf{Cat}}
\newcommand{\Catm}{\Cat_{\M}}
\newcommand{\Catmab}{\Catm^{\{a,b\}}}
\newcommand{\Catmc}{\Catm^{\fC}}
\newcommand{\Catmd}{\Catm^{\fD}}
\newcommand{\Catn}{\Cat_{\N}}

\newcommand{\set}{\mathsf{Set}}
\newcommand{\setop}{\set^{\mathrm{op}}}

\newcommand{\smod}{\mathsf{SMod}}
\newcommand{\sset}{\mathsf{SSet}}

\newcommand{\symseq}{\mathsf{SymSeq}}
\newcommand{\symseqc}{\symseq_{\fC}}
\newcommand{\symseqcm}{\symseqc(\calm)}

\newcommand{\alg}{\mathsf{Alg}}

\newcommand{\sigmaop}{\Sigma^{\smallop}}

\newcommand{\sigmaopb}{\Sigma^{\smallop}_{\smallbrb}}
\newcommand{\sigmaopc}{\Sigma^{\smallop}_{\smallbrc}}

\newcommand{\operad}{\mathsf{Operad}}
\newcommand{\operadm}{\operad(\M)}
\newcommand{\operadmop}{\operad(\M)^{\mathrm{op}}}

\newcommand{\operadc}{\operad^{\fC}}
\newcommand{\operadcm}{\operadc(\M)}
\newcommand{\operadscm}{\operad^{\sofc}(\M)}

\newcommand{\pofc}{\Sigma_{\frakC}}
\newcommand{\pofcop}
{\pofc^{\scalebox{.6}{$\mathrm{op}$}}}

\newcommand{\prof}{\mathsf{Prof}}
\newcommand{\profc}{\prof(\fC)}
\newcommand{\profcsq}{\profc^{\times 2}}
\newcommand{\profcc}{\profc \times \fC}
\newcommand{\profd}{\prof(\fD)}
\newcommand{\profdd}{\profd \times \fD}

\newcommand{\sofc}{S(\fC)}
\newcommand{\sofd}{S(\fD)}
\newcommand{\profscsc}{\prof(\sofc) \times \sofc}

\newcommand{\smallprof}[1]
{\raisebox{.05cm}{\scalebox{0.8}{#1}}}

\newcommand{\aiai}
{\smallprof{$\binom{a_i}{a_i}$}}

\newcommand{\bonebnaoneam}
{\smallprof{$\binom{b_{[1,n]}}{a_{[1,m]}}$}}
\newcommand{\ibonebnaoneam}
{\smallprof{$\binom{ib_{[1,n]}}{ia_{[1,m]}}$}}
\newcommand{\hibonebnaoneam}
{\smallprof{$\binom{hib_{[1,n]}}{hia_{[1,m]}}$}}

\newcommand{\bjbj}
{\smallprof{$\binom{b_{j}}{b_{j}}$}}

\newcommand{\cjbrbj}
{\smallprof{$\binom{c_j}{[\ub_j]}$}}

\newcommand{\singledbrb}
{\smallprof{$\binom{d}{[\ub]}$}}

\newcommand{\dc}
{\smallprof{$\binom{\ud}{\uc}$}}
\newcommand{\dcsingle}
{\smallprof{$\binom{d}{c}$}}
\newcommand{\duc}
{\smallprof{$\binom{d}{\uc}$}}
\newcommand{\fduc}
{\smallprof{$\binom{fd}{f\uc}$}}
\newcommand{\hduc}
{\smallprof{$\binom{hd}{h\uc}$}}
\newcommand{\fdfuc}
{\smallprof{$\binom{f_0d}{f_0\uc}$}}

\newcommand{\singledbrc}
{\smallprof{$\binom{d}{[\uc]}$}}

\newcommand{\fdc}
{\smallprof{$\binom{f\ud}{f\uc}$}}
\newcommand{\fzerodc}
{\smallprof{$\binom{f_0\ud}{f_0\uc}$}}

\newcommand{\tzerotonetn}{\smallprof{$\binom{t_0}{t_1,\ldots,t_n}$}}
\newcommand{\tzerotzero}{\smallprof{$\binom{t_0}{t_0}$}}
\newcommand{\fzerots}{\smallprof{$\binom{f_0t_0}{f_0t_1,\ldots,f_0t_n}$}}

\newcommand{\smallbr}[1]
{\raisebox{.03cm}{\scalebox{0.5}{#1}}}

\newcommand{\smallbrb}{\smallbr{$[\ub]$}}

\newcommand{\smallbrc}{\smallbr{$[\uc]$}}

\newcommand{\sigmabra}{\Sigma_{\smallbr{$[\ua]$}}}

\newcommand{\sigmabrb}{\Sigma_{\smallbr{$[\ub]$}}}

\newcommand{\sigmabrbj}{\Sigma_{\smallbr{$[\ub_j]$}}}

\newcommand{\sigmabrc}{\Sigma_{\smallbr{$[\uc]$}}}

\newcommand{\sigmabrbop}{\sigmabrb^{\smallop}}
\newcommand{\sigmabrbjop}{\sigmabrbj^{\smallop}}

\newcommand{\sigmabrcop}{\sigmabrc^{\smallop}}

\newcommand{\sigmabrcopd}{\sigmabrcop \times \{d\}}

\newcommand{\sigmacop}{\pofcop}

\newcommand{\sigmacopc}{\sigmacop \times \fC}

\newcommand{\dbrch}{([\uc];d)}

\renewcommand{\lim}{\mathsf{lim}\,}

\newcommand{\Hom}{\mathsf{Hom}}
\newcommand{\Homm}{\Hom_{\M}}

\DeclareMathOperator{\Id}{Id}
\DeclareMathOperator{\Kan}{\mathsf{Kan}}

\DeclareMathOperator{\Ob}{Ob}

\begin{document}

\title{Dwyer-Kan Homotopy Theory of Algebras over Operadic Collections}


\author{Donald Yau}
\address{The Ohio State University at Newark \\ Newark, OH}
\email{yau.22@osu.edu}

\begin{abstract}
Over suitable monoidal model categories, we construct a Dwyer-Kan model category structure on the category of algebras over an augmented operadic collection.  As examples we obtain Dwyer-Kan model category structure on the categories of enriched wheeled props, wheeled properads, and wheeled operads, among others.  This result extends known model category structure on the categories of operads, properads, and props enriched in simplicial sets and other monoidal model categories.  We also show that our Dwyer-Kan model category structure is well behaved with respect to simultaneous changes of the underlying monoidal model category and the augmented operadic collection.
\end{abstract}

\maketitle

\tableofcontents

\section{Introduction}

This paper is about higher categorical properties of operad-like algebraic structures.  An important component in the study of higher algebraic structure is a suitable model category structure on the category of small categories enriched in simplicial sets, or simplicial categories for short.  Such a model category structure, called the Dwyer-Kan model category structure, was first proved in \cite{bergner}.  For more general base monoidal model categories, the Dwyer-Kan model category structure on small enriched categories was obtained in \cite{lurie,muro15}.  A closely related model category structure on small enriched categories appeared in \cite{bm13}.

Operads are generalizations of categories in which arrows are allowed to have any finite number of inputs.  They are sometimes called small symmetric multicategories; see \cite{yau16} for a gentle introduction to operads.  In the algebraic setting, algebras over operads include associative algebras, commutative algebras, Lie algebras, Poisson algebras, and diagrams of such algebras, among others.  In the topological setting, algebras over operads include iterated loop spaces, $E_n$-algebras, and other structured ring spectra.  Going up from categories to operads, there is an analogous Dwyer-Kan model category structure on the category of simplicial operads, i.e., operads with a set of colors, which are allowed to vary, enriched in simplicial sets \cite{cm13,robertson}.  For more general base monoidal model categories, such a  Dwyer-Kan model category structure on enriched operads was obtained in \cite{cav14}.  

Props are generalizations of operads in the sense that arrows are allowed to have any finite number of inputs and outputs; see \cite{jy2} for an in-depth discussion of various cousins of operads and props, including their wheeled generalizations.  Props can model both multiplicative and comultiplicative structure simultaneously, including bialgebras and Lie bialgebras.  The Dwyer-Kan model category structure for simplicial props \cite{hr16} and simplicial properads \cite{hry.properad} are known to exist.  For more general base monoidal model categories, the Dwyer-Kan model category structure on enriched props was proved in \cite{cav15}.

The first main objective of this paper is to prove the existence of the Dwyer-Kan model category structure for a much larger class of enriched operad-like structure in a unified manner.  This includes not only enriched operads, properads, and props, but also enriched wheeled operads, wheeled properads, and wheeled props \cite{jy2}.  Wheeled versions of operads, properads, and props are important in topological conformal field theory, deformation quantization, Poisson geometry, and graph cohomology, among others \cite{mms,merkulov}.  The wheeled part of these algebraic structures refers to a structure map called the \emph{contraction}, which in specific examples is actually the trace map.

For a fixed color set $\fC$, there is a colored operad whose algebras are exactly the $\fC$-colored wheeled props \cite{jy2}.  Similar colored operads exist for the other operad-like structure with a fixed color set.  Allowing the color set to vary, we arrive at the concept of an \emph{operadic collection} $(S,\sO)$ (Def. \ref{def:operadic.collection}).  Here $S$ parametrizes the possible numbers of inputs and outputs, and $\sO$ assigns to each set an operad in a functorial way.  For example, there is an operadic collection whose algebras are exactly the enriched wheeled props, and likewise for the other enriched operad-like structure (Example \ref{ex:gprops}).  The category of algebras over an operadic collection is actually the Grothendieck construction of some functor; see \eqref{op.coll.obar} and Remark \ref{rk:integral.model}.

An operadic collection is \emph{augmented} if, roughly speaking, it is nicely related to the operadic collection whose algebras are enriched categories with extra entries (Def. \ref{def:aug.opcoll}).  An augmented operadic collection is \emph{admissible} if its category of algebras admits the Dwyer-Kan model category structure (Def. \ref{def:opcoll.admissible}).  The following observation is our first main result. It will appear as Theorem \ref{algso.dkmodel} below.  

\begin{theorem}\label{first.main.theorem}
Every augmented operadic collection in a convenient model category (Def. \ref{def:convenient}) is admissible.
\end{theorem}

For example, as we will see in Example \ref{ex:gprop.dwyer.kan}, using appropriately chosen augmented operadic collections, we obtain Dwyer-Kan model category structure on the categories of enriched (wheeled) props, (wheeled) properads, and (wheeled) operads.  For enriched operads, properads, and props, our results recover known model category structure mentioned above.  For enriched wheeled operads, wheeled properads, and wheeled props, our results are new.  Furthermore, this Dwyer-Kan model category structure is cofibrantly generated in which trivial fibrations and fibrant objects are easy to describe.  In fact, a map is a trivial fibration if and only if it is entrywise a trivial fibration that is surjective on colors.  An object is fibrant if and only if it is entrywise fibrant.

In the previous theorem, due to its generality, the underlying category has to be suitably restricted to what we call a convenient model category (Def. \ref{def:convenient}).  Examples include the familiar model categories of simplicial sets, of chain complexes and simplicial modules over a field of characteristic zero, and of small categories with the folk model structure.

Our strategy for proving Theorem \ref{first.main.theorem} is different from the existing strategy for proving the existence of the Dwyer-Kan model category structure on enriched categories, operads, and props mentioned above.  Given an augmented operadic collection $(S,\sO)$ with augmentation $\alpha$, there are two adjunctions
\[\nicexy{\Catm \ar@<2pt>[r]^-{\text{add $\varnothing$}} &\algsass \ar@<2pt>[l]^-{(-)^{\cat}} \ar@<2pt>[r]^-{\alpha_!} & \algso \ar@<2pt>[l]^-{\alpha^\ast}}\]
in which $\M$ is the underlying category and $\Catm$ is the category of small $\M$-enriched categories (Example \ref{ex:enriched.cat}).  It is important to note that the Dwyer-Kan weak equivalences and fibrations in $\algso$ cannot be defined by the forgetful functor to $\Catm$ because the latter does not have enough entries. The middle category $\algsass$ is what we call the category of \emph{$\M$-enriched categories with $S$-entries} (Example \ref{ex:enriched.cat.extra}).  It is an interpolation between the other two categories such that the Dwyer-Kan weak equivalences and fibrations in $\algso$ can be defined in $\algsass$.  The adjunction on the left consists of adding initial objects for the extra entries and of concentrating on the categorical part \eqref{enriched.cat.adjoint}.  The adjunction on the right is induced by the augmentation $\alpha$ (Prop. \ref{opcoll.map.leftadj}).  As a preliminary step, in Theorem \ref{algass.dk.model} we will show that, under suitable assumptions on $\M$, $\algsass$ admits the Dwyer-Kan model category structure, which is \emph{not} lifted from $\Catm$.  The Dwyer-Kan model category structure in Theorem \ref{first.main.theorem} is then obtained from $\algsass$ by lifting via the adjunction $(\alpha_!,\alphastar)$.

Our second main result is about the homotopy invariance of the Dwyer-Kan homotopy theory in the previous theorem.  The idea is that, under suitable conditions, a Quillen equivalence between the underlying categories should lift to a Quillen equivalence between the algebra categories.  In the case of monoids, such a homotopy invariance result goes back to \cite{ss03} (3.12).  For small enriched categories, a result along the same lines is \cite{muro15} (1.4).  For algebras over colored operads with a fixed color set, such a homotopy invariance result is in \cite{white-yau2}; see Theorem \ref{fixed.color.lifting}.  Here we extend it to the categories of algebras over augmented operadic collections.  The following observation is our second main result.  It will appear as Theorem \ref{lifting.qeq} below.

\begin{theorem}
Suppose:
\begin{itemize}
\item $L : \M \adjoint \N : R$ is a weak symmetric monoidal Quillen equivalence (Def. \ref{quillen.equivalence}) between convenient model categories (Def. \ref{def:convenient}) such that every generating cofibration in $\M$ has cofibrant domain.
\item $(S,\sO,\alphaone)$ is an augmented operadic collection in $\M$ (Def. \ref{def:aug.opcoll}), and $(S,\sP,\alphatwo)$ is an augmented operadic collection in $\N$ with the same $S$.
\item $\alpha : (S,\sO) \to (S,R\sP)$ is a map of operadic collections in $\M$ that is compatible with the augmentations. 
\item For each set $\fC$, the entrywise adjoint of the map $\alphac : \osubc \to R\psubc$ is an entrywise weak equivalence in $\N$.
\end{itemize}
Suppose further that one of the following two conditions holds:
\begin{enumerate}
\item $(L,R)$ is a nice Quillen equivalence (Def. \ref{def:nice.qeq}), and both operadic collections $(S,\sO)$ and $(S,\sP)$ are entrywise cofibrant (Def. \ref{def:sigma.cofibrant}).
\item Both $(S,\sO)$ and $(S,\sP)$ are $\Sigma$-cofibrant.
\end{enumerate}
Then there is an induced Quillen equivalence
\[\nicexy{\algso \ar@<2.5pt>[r]^-{\Lbar} & \algsp. \ar@<2.5pt>[l]^-{\Ralpha}}\]
\end{theorem}

For example, as we will explain in Section \ref{operad.like.lifting}, if $(L,R)$ is a nice Quillen equivalence, then there are induced Quillen equivalences between categories of enriched (wheeled) props, (wheeled) properads, and (wheeled) operads.  Examples of nice Quillen equivalences, which appear in (1) above, include the Dold-Kan correspondence for a field of characteristic zero and the analogous adjunction for reduced rational simplicial and chain complexes of Lie algebras \cite{quillen} (p.211).

The rest of this paper is organized as follows.  In Section \ref{sec:colored} we recall some basic concepts regarding colored operads.  In Section \ref{sec:operadic.collections} we define operadic collections, maps between them, and their algebras, and discuss a series of examples.  In Section \ref{sec:dk.enriched.cat} we recall the Dwyer-Kan homotopy theory of small enriched categories from \cite{bm13,muro15}.  In Section \ref{sec:dk.enrichedcat.extra} we extend this to the Dwyer-Kan homotopy theory of small enriched categories with extra entries.  Our two main results are proved in Section \ref{sec:alg.opcoll} and Section \ref{sec:lifting.quillen}, respectively.

\subsection*{Acknowledgment}
The author would like to thank David White for reading an earlier draft of this paper and for his many useful comments.  The author would also like to thank Philip Hackney for pointing out an inaccuracy in an earlier version of this paper.

\section{Colored Operads with Fixed Colors}
\label{sec:colored}

In this section we recall some results regarding colored operads for a fixed color set.  Throughout this section $(\M, \otimes, \tensorunit)$ denotes a bicomplete (i.e., has all small limits and colimits) symmetric monoidal closed category with initial object $\varnothing$.   The monoidal unit will be written as $\tensorunitm$ if we need to emphasize $\M$.

\subsection{Colors and Profiles}

Here we recall from \cite{jy2} some notations regarding colors that are needed to talk about colored objects.

\begin{definition}[Colored Objects]
\label{def:profiles}
Fix a non-empty set $\fC$, whose elements are called \emph{colors}.
\begin{enumerate}
\item
A \emph{$\fC$-profile} is a finite sequence of elements in $\fC$, say,
\[\uc = (c_1, \ldots, c_m)\]
with each $c_i \in \fC$.  The set of $\fC$-profiles is denoted $\profc$. The empty $\fC$-profile is denoted $\emptyset$, which is not to be confused with the initial object in $\calm$.  Write $|\uc|=m$ for the \emph{length} of a profile $\uc$.
\item
An object in the product category $\prod_{\fC} \calm = \calm^{\fC}$ is called a \emph{$\fC$-colored object in $\calm$}, and similarly for a map of $\fC$-colored objects.  A typical $\fC$-colored object $X$ is also written as $\{X_a\}$ with $X_a \in \calm$ for each color $a \in \fC$.
\end{enumerate}
\end{definition}

Next we define the colored version of a $\Sigma$-object, also known as a symmetric sequence.

\begin{definition}[Colored Symmetric Sequences]
\label{def:colored-sigma-object}
Fix a non-empty set $\fC$.
\begin{enumerate}
\item
If $\ua = (a_1,\ldots,a_m)$ and $\ub$ are $\fC$-profiles, then a \emph{left permutation} $\sigma : \ua \to \ub$ is a permutation $\sigma \in \Sigma_{|\ua|}$ such that
\[\sigma\ua = (a_{\sigma^{-1}(1)}, \ldots , a_{\sigma^{-1}(m)}) = \ub\]
This necessarily implies $|\ua| = |\ub| = m$.
\item
The \emph{groupoid of $\fC$-profiles}, with left permutations as the isomorphisms, is denoted by $\pofc$.  The opposite groupoid $\pofcop$ is regarded as the groupoid of $\fC$-profiles with \emph{right permutations}
\[\ua\sigma = (a_{\sigma(1)}, \ldots , a_{\sigma(m)})\]
as isomorphisms.
\item
The \emph{orbit} of a profile $\ua$ is denoted by $[\ua]$.  The maximal connected sub-groupoid of $\pofc$ containing $\ua$ is written as $\sigmabra$.  Its objects are the left permutations of $\ua$.  There is a decomposition
\begin{equation}
\label{pofcdecomp}
\pofc \cong \coprod_{[\ua] \in \pofc} \sigmabra,
\end{equation}
where there is one coproduct summand for each orbit $[\ua]$ of a $\fC$-profile.  By $[\ua] \in \pofc$ we mean that $[\ua]$ is an orbit in $\pofc$.
\item
Define the diagram category
\begin{equation}\label{symmetric.sequences}
\symseqcm = \calm^{\sigmacopc},
\end{equation}
whose objects are called \emph{$\fC$-colored symmetric sequences}.  By the decomposition \eqref{pofcdecomp}, there is a decomposition
\[\symseqcm \cong 
\prod_{\dbrch \in \sigmacopc} \calm^{\sigmabrcopd},\]
where $\sigmabrcopd \cong \sigmaopc$.  
\item
For $X \in \symseqcm$, we write
\[X\singledbrc \in \calm^{\sigmabrcopd} \cong \calm^{\sigmabrcop}\]
for its $\dbrch$-component.  For $\duc \in \profcc$ (i.e., $\uc$ is a $\fC$-profile and $d \in \fC$), we write
\[X\duc \in \calm\]
for the value of $X$ at $\duc$.  In what follows, we think of the profile $\uc \in \profc$ as parametrizing the inputs, while $d \in \fC$ is the color of the output.
\end{enumerate}
\end{definition}

\subsection{Colored Circle Product}

We will define $\fC$-colored operads as monoids with respect to the $\fC$-colored circle product.  To define the latter, we need the following definition.

\begin{definition}
\label{def:tensorover}
Suppose $\cald$ is a small groupoid, $X \in \calm^{\caldop}$, and $Y \in \calm^{\cald}$.  Define the object $X \otimes_{\cald} Y \in \calm$ as the colimit of the composite
\[\nicexy{\cald \ar[r]^-{\cong \Delta} 
& \caldop \times \cald \ar[r]^-{(X,Y)}
& \calm \times \calm \ar[r]^-{\otimes}
& \calm,}\]
where the first map is the diagonal map followed by the isomorphism $\cald \times \cald \cong \caldop \times \cald$.
\end{definition}

We will mainly use the construction $\otimes_{\cald}$ when $\cald$ is the finite connected groupoid $\sigmabrc$ for some orbit $[\uc] \in \pofc$.

\begin{convention}
For an object $A \in \calm$, $A^{\otimes 0}$ means $\tensorunit$, the monoidal unit in $\calm$.
\end{convention}

\begin{definition}[Colored Circle Product]
\label{def:colored-circle-product}
Suppose $X,Y  \in \symseqcm$, $d \in \fC$, $\uc = (c_1,\ldots,c_m) \in \profc$, and $[\ub] \in \pofc$ is an orbit.
\begin{enumerate}
\item
Define the object
\[Y^{\uc} \in \calm^{\pofcop} \cong \prod_{[\ub] \in \pofc} \calm^{\sigmabrbop}\]
as having the $[\ub]$-component
\[Y^{\uc}([\ub]) =
\coprod_{\substack{\{[\ub_j] \in \pofc\}_{1 \leq j \leq m} \,\mathrm{s.t.} \\
[\ub] = [(\ub_1,\ldots,\ub_m)]}} 
\Kan^{\sigmabrbop} 
\left[\bigotimes_{j=1}^m Y \cjbrbj\right] 
\in \calm^{\sigmabrbop}.\]
The above left Kan extension is defined as
\[\nicexy{\prod_{j=1}^m \sigmabrbjop 
\ar[d]_-{\mathrm{concatenation}} 
\ar[rr]^-{\prod Y \binom{c_j}{-}} 
&& \calm^{\times m} \ar[d]^-{\otimes}\\
\sigmabrbop \ar[rr]_-{\Kan^{\sigmabrbop}\left[\otimes Y(\vdots)\right]}^-{\mathrm{left ~Kan~ extension}}  && \calm.}\]
\item
By allowing left permutations of $\uc$ above, we obtain
\[Y^{[\uc]} \in \calm^{\pofcop \times \sigmabrc} \cong \prod_{[\ub] \in \pofc} \calm^{\sigmabrbop \times \sigmabrc}\]
with components
\[Y^{[\uc]}([\ub]) \in \calm^{\sigmabrbop \times \sigmabrc}.\]
\item 
The \emph{$\fC$-colored circle product}
\[X \circ Y \in \symseqcm\]
is defined to have components
\[(X \circ Y)\singledbrb 
= \coprod_{[\uc] \in \pofc} 
X\singledbrc \tensorover{\sigmabrc} Y^{[\uc]}([\ub]) \in \calm^{\sigmaopb \times \{d\}}\]
for $d \in \fC$ and $[\ub] \in \pofc$, where the coproduct is indexed by all the orbits in $\pofc$.
\end{enumerate}
\end{definition}

\begin{proposition}
\label{circle-product-monoidal}
With respect to the $\fC$-colored circle product $\circ$, $\symseqcm$ is a monoidal category.
\end{proposition}

\begin{proof}
This is 3.2.18 in \cite{white-yau}.  In the one-colored case, this is first proved in \cite{rezk.thesis} (2.2.6) and also in \cite{harper-jpaa} (4.21).
\end{proof}

\begin{definition}\label{def:colored-operad}
For a set $\fC$ of colors, the category of \emph{$\fC$-colored operads} in $\calm$, denoted $\operadcm$, is defined as the category of monoids \cite{maclane} (VII.3) in the monoidal category $\bigl(\symseqcm, \comp\bigr)$.
\end{definition}

The reader is referred to \cite{yau16} (Ch. 11) for an explicit description of the structure maps of a $\fC$-colored operad.

\subsection{Algebras over a Colored Operad}

We will regard $\calmc$ as a full subcategory of $\symseqcm$ \eqref{symmetric.sequences} by identifying $\{X_c\}_{c\in \fC} \in \calmc$ with the $\fC$-colored symmetric sequence with entries
\[X\duc = \begin{cases} X_d & \text{if $\uc = \varnothing$},\\
\varnothing & \text{otherwise}.\end{cases}\]

\begin{definition}\label{colored-operad-algebra}
Suppose $\sO$ is a $\fC$-colored operad.  The category of algebras over the monad \cite{maclane} (VI.2)
\[\sO \comp - : \calmc \to \calmc\]
is denoted by $\alg(\sO)$, whose objects are called \emph{$\sO$-algebras} in $\calm$ \cite{yau16} (Def. 13.2.3).
\end{definition}

\begin{definition}\label{def:asubc}
Suppose $A = \{A_c\}_{c\in \fC} \in \calm^{\fC}$ and $\uc = (c_1,\ldots,c_n) \in \pofc$ with orbit $[\uc]$.  Define the object
\[A_{\uc} = \bigotimes_{i=1}^n A_{c_i} = A_{c_1} \otimes \cdots \otimes A_{c_n} \in \calm\]
and the diagram $A_{\smallbrc} \in \calm^{\sigmabrc}$ with values
\[A_{\smallbrc}(\uc') = A_{\uc'}\]
for each $\uc' \in [\uc]$.  All the structure maps in the diagram $A_{\smallbrc}$ are given by permuting the factors in $A_{\uc}$.
\end{definition}

There is a free-forgetful adjoint pair
\[\nicexy{\calmc \ar@<2pt>[r]^-{\sO \comp -} 
& \alg(\sO) \ar@<2pt>[l]^-{U}}\]
for each $\fC$-colored operad $\sO$.  When drawing an adjunction, our convention is to draw the left adjoint on top.

\begin{proposition}\label{algebra-bicomplete}
Suppose $\sO$ is a $\fC$-colored operad in $\M$.  Then the category $\alg(\sO)$ has all small limits and colimits, with reflexive coequalizers and filtered colimits preserved and created by the forgetful functor $\alg(\sO) \to \calmc$.
\end{proposition}

\begin{proof}
This is 4.2.1 in \cite{white-yau}.  In the one-colored case, this is first proved in \cite{rezk.thesis} (2.3.5) and also in \cite{harper-jpaa} (5.15).
\end{proof}

\section{Operadic Collections}
\label{sec:operadic.collections}

In this section we define operadic collections, our main objects of study, and discuss a series of examples of interest.  Throughout this section $(\M, \otimes, \tensorunit)$ denotes a bicomplete symmetric monoidal closed category with initial object $\varnothing$.  Recall that $\profc$ is the set of $\fC$-profiles for a set $\fC$ (Def. \ref{def:profiles}).

\subsection{Colored Operads with Varying Colors}

\begin{definition}\label{def:operad.change.color}
Suppose $\sP$ is a $\fD$-colored operad in $\M$ and $f : \fC \to \fD \in \set$.  Then $f^*\sP$ is the $\fC$-colored operad with entries
\[(f^*\sP)\duc = \sP\fduc\]
for $\duc \in \profcc$ and structure maps uniquely determined by those of $\sP$.  Here $\fduc \in \profdd$ is the result of applying $f$ entrywise to $\duc$. 
\end{definition}

\begin{definition}\label{def:all.operads}
Define the category of \emph{colored operads in $\M$}, denoted $\operadm$, as follows.
\begin{enumerate}
\item An object in $\operadm$ is a pair $(\fC,\sO)$ consisting of:
\begin{itemize}
\item a set $\fC$, called the \emph{color set}, and
\item a $\fC$-colored operad $\sO$ in $\M$ (Def. \ref{def:colored-operad}).
\end{itemize}
\item A map $f : (\fC,\sO) \to (\fD,\sP)$ in $\operadm$ consists of two maps $(f_0,f_1)$ such that:
\begin{itemize}
\item $f_0 : \fC \to \fD$ is a map of color sets.
\item For each $\duc \in \profcc$, $f_1$ is an entry map
\begin{equation}\label{operad.map.entry}
\nicexy{\sO\duc \ar[r]^-{f_1} & \sP\fdfuc \in \M}.
\end{equation}
It is required that $f_1 : \sO \to f_0^*\sP$ be a map of $\fC$-colored operads, where the $\fC$-colored operad $f_0^*\sP$ is as in Def. \ref{def:operad.change.color}.
\end{itemize}
In what follows we will usually write both $f_0$ and $f_1$ as $f$.
\end{enumerate}
\end{definition}

\begin{remark}\label{rk:operad.map.factor}
Each map $f : (\fC,\sO) \to (\fD,\sP)$ in $\operadm$ has a canonical factorization
\[\nicexy{(\fC,\sO) \ar[r]^-{g} \ar `u[rr] `[rr]|-{f} [rr] & (\fC,f^*\sP) \ar[r]^-{h} & (\fD,\sP)}\]
such that:
\begin{enumerate}
\item $g$ is the identity map of $\fC$ on colors, and each entry map of $g$ is an entry map of $f$:
\[\nicexy{\sO\duc \ar[r]^-{g \,=\, f} & (f^*\sP)\duc = \sP\fduc.}\]
\item $h$ is $f : \fC \to \fD$ on colors, and each entry map of $h$ is the identity map:
\[\nicexy{(f^*\sP)\duc = \sP\fduc \ar[r]^-{=} & \sP\fduc = \sP\hduc.}\]
\end{enumerate}
This implies that properties of colored operads can often be deduced from the fixed color set case.
\end{remark}

\begin{remark}
In the literature, a colored operad in $\M$ is also called a small symmetric multicategory enriched in $\M$.  When $\M$ is the category $\sset$ of simplicial sets, $\operad(\sset)$ is called the category of small multicategories enriched in simplicial sets in \cite{robertson} and the category of simplicial operads in \cite{cm13}.
\end{remark}

\subsection{Operadic Collections and Algebras}

\begin{notation}
Suppose $\fC$ is a set and $S \subseteq \bbN^{\times 2} = \{0,1,2,\ldots\}^{\times 2}$.  Define the set
\begin{equation}\label{sofc}
\sofc = \Bigl\{\dc \in \profcsq : (|\uc|, |\ud|) \in S\Bigr\}.
\end{equation}
In writing $\dc \in \profcsq$, our convention is to regard $\uc$ (resp., $\ud$) as in the first (resp., second) copy of $\profc$, parametrizing the inputs (resp., outputs).
\end{notation}

\begin{example}
\begin{enumerate}
\item $\bbNsq(\fC) = \profcsq$.
\item If $S = \bbN \times \{1\}$, then $\sofc = \profcc$.
\item If $S = \{(1,1)\}$, then $\sofc = \csquare$.
\end{enumerate}
\end{example}

\begin{definition}\label{def:operadic.collection}
An \emph{operadic collection in $\M$} is a pair $(S, \sO)$ consisting of 
\begin{itemize}
\item a subset $S \subseteq \bbN^{\times 2}$ and
\item a functor $\sO : \set \to \operadm$
\end{itemize}
such that the following two conditions hold.  For a set $\fC$, we will write $\sO(\fC)$ as $\osubc$.
\begin{enumerate}
\item For each set $\fC$, $\osubc$ is an $\sofc$-colored operad in $\M$ (Def. \ref{def:colored-operad}).
\item For each map $f : \fC \to \fD \in \set$, the map $\sO_f : \osubc \to \osubd \in \operadm$ on color sets is given by 
\[\nicexy{\sofc \ni \dc \ar@{|->}[r] & f\dc = \fdc \in \sofd,}\]
where $f\uc \in \profd$ is the result of applying $f$ entrywise to $\uc \in \profc$ and likewise for $f\ud$.
\end{enumerate}
If $S$ is clear from the context, then we will omit it from the notation.
\end{definition}

\begin{definition}\label{def:operadic.collection.algebra}
Suppose $(S,\sO)$ is an operadic collection in $\M$.  
\begin{enumerate}
\item An \emph{$(S,\sO)$-algebra} is a pair $(\fC,A)$ consisting of a set $\fC$, called the \emph{color set}, and an $\osubc$-algebra $A$ (Def. \ref{colored-operad-algebra}).
\item Suppose $(\fC,A)$ and $(\fD,B)$ are $(S,\sO)$-algebras.  A \emph{map} $f : (\fC,A) \to (\fD,B)$ of $(S,\sO)$-algebras consists of a pair $(f_0,f_1)$ of maps such that
\begin{itemize}
\item $f_0 : \fC \to \fD$ is a map of color sets and
\item $f_1 : A \to f_0^*B$ is a map of $\osubc$-algebras.
\end{itemize}
Here $f_0^*B$ is the $\osubc$-algebra with entries
\[(f_0^*B)\dc = B\fzerodc\]
for $\dc \in \sofc$.  Its $\osubc$-algebra structure maps are the composites
\begin{equation}\label{pullback.algebra}
\nicexy@C+.5cm{\osubc\tzerotonetn \otimes (f_0^*B)(t_1) \otimes \cdots \otimes (f_0^*B)(t_n) \ar[d]_-{(\sO_{f_0},\Id)} \ar[r] & (f_0^*B)(t_0)\\ 
\osubd\fzerots \otimes B(f_0t_1) \otimes \cdots \otimes B(f_0t_n) \ar[r]^-{\osubd-\mathrm{algebra}}_-{\mathrm{structure}} & B(f_0t_0) \ar[u]_-{=}}
\end{equation}
in which $n\geq 0$, each $t_i \in \sofc$, and $f_0t_i \in \sofd$ is the result of applying $f_0$ entrywise to $t_i$.  In what follows, we will usually write both $f_0$ and $f_1$ as $f$.
\item The category of $(S,\sO)$-algebras is denoted $\alg(S,\sO)$, or $\algo$ if $S$ is clear from the context.
\end{enumerate}
\end{definition}

\begin{remark}\label{rk:grothendieck.con}
\begin{enumerate}
\item In \cite{cav15} (3.2) there is a definition similar to that of $\algso$.
\item In Definition \ref{def:operadic.collection.algebra}, $f_0^*B$ is the image of $B$ under the functor
\[\nicexy{\alg(\osubc) & \alg(\osubd) \ar[l]_-{\sO_{f_0}^*}}\]
induced by the map $\sO_{f_0} : \osubc \to \osubd \in \operadm$, 
with entries and structure maps defined as above.  Furthermore, this functor admits a left adjoint.
\item There is a functor
\[\nicexy{\operadmop \ar[r]^-{\alg} & \Cat}\]
that sends a colored operad $(\fC,\sP) \in \operadm$ to the category $\algp$ of $\sP$-algebras.  For a map $f : (\fC,\sP) \to (\fD,\sQ) \in \operadm$, the functor $f^* : \algq \to \algp$ is defined similarly to $f_0^*B$ above.   Given an operadic collection $(S,\sO)$ in $\M$, consider the composite
\begin{equation}\label{op.coll.obar}
\nicexy{\setop \ar[r]^-{\sOop} \ar `u[rr] `[rr]|-{\Obar} [rr]& \operadmop \ar[r]^-{\alg} & \Cat.}
\end{equation}
Then the category $\alg(S,\sO)$ in Def. \ref{def:operadic.collection.algebra} coincides with the Grothendieck construction of $\Obar$
\cite{maclane-moerdijk} (p.41-44).
\end{enumerate}
\end{remark}

\begin{proposition}\label{algso.bicomplete}
For each operadic collection $(S,\sO)$ in a bicomplete symmetric monoidal closed category $\M$, the category $\algso$ is bicomplete.
\end{proposition}

\begin{proof}
For each set $\fC$, the category $\alg(\osubc)$ has all small limits and colimits (Prop. \ref{algebra-bicomplete}).  Since $\set$ has all small (co)limits, so does $\algso$.  Alternatively, one could simply apply \cite{hp} (2.4.4), using the fact that $\algso$ is the Grothendieck construction of $\Obar$ \eqref{op.coll.obar}.
\end{proof}

The following observation is an immediate consequence of Def. \ref{def:operadic.collection.algebra}.

\begin{proposition}\label{algebra.map.factorization}
Suppose $(S,\sO)$ is an operadic collection in a bicomplete symmetric monoidal closed category $\M$, and $f : (\fC,A) \to (\fD,B) \in \algso$ is a map of $(S,\sO)$-algebras.  Then there is a canonical factorization
\begin{equation}\label{f.factorization}
\nicexy{(\fC,A) \ar[r]^-{g} \ar `u[rr] `[rr]|-{f} [rr] & (\fC,f^*B) \ar[r]^-{h} & (\fD,B)}
\end{equation}
in $\algso$ such that:
\begin{enumerate}
\item $f^*B \in \algosubc$ has entries
\[(f^*B)_{t} = B_{f(t)}\]
for $t \in \sofc$ and $\osubc$-algebra structure maps as in \eqref{pullback.algebra}.
\item On colors $g$ is the identity map on $\fC$, and $g : A \to f^*B \in \algosubc$ has entry maps
\[\nicexy{A_t \ar[r]^-{g \,=\, f} & (f^*B)_t = B_{f(t)}}\]
the same as $f$ for $t \in \sofc$.
\item On colors $h = f : \fC \to \fD$.  Every entry map of $h$,
\[\nicexy{(f^*B)_t = B_{f(t)} \ar[r]^-{h \,=\, \Id} & B_{f(t)}}\]
for $t \in \sofc$, is the identity map.
\end{enumerate}
\end{proposition}

\subsection{Examples of Operadic Collections}\label{sec:ex.op.coll}

\begin{example}[Underlying Objects]\label{ex:initial.collection}
Suppose $S \subseteq \bbNsq$.  For each set $\fC$, write $\sI_{\fC}$ for the initial $\sofc$-colored operad in $\M$ \cite{yau16} (11.4.1).  Its entries are either the initial object $\varnothing$ or the monoidal unit $\tensorunit$.  There is a unique operadic collection
\[(S,\sI)\]
such that for each map $f : \fC \to \fD \in \set$, the map
\[\nicexy{\sI_{\fC} \ar[r]^-{\sI_f} & \sI_{\fD} \in \operadm}\]
is entrywise either the identity map of the monoidal unit in $\M$ or the unique map from $\varnothing$ to $\varnothing$ or $\tensorunit$.

Let us write $\mtosdash$ for the category $\alg(S,\sI)$.  Explicitly, an object in $\mtosdash$ is a pair $(\fC,A)$ consisting of a set $\fC$ and an object $A \in \M^{\sofc}$, so $A$ has entries $A\dc \in \M$ for $\dc \in \sofc$.  A map
\[\nicexy{(\fC,A) \ar[r]^-{f} & (\fD,B) \in \mtosdash}\]
consists of a map $f : \fC \to \fD$ of color sets and an entry map 
\[\nicexy{A\dc \ar[r]^-{f} & B\fdc \in \M}\]
for each $\dc \in \sofc$.
\end{example}

\begin{example}[$\M$-Enriched Categories]\label{ex:enriched.cat}
For each set $\fC$, suppose $\Asc$ is the $\csquare$-colored operad such that $\alg(\Asc)$ is the category of $\M$-enriched categories with object set $\fC$ and object-preserving functors \cite{yau16} (14.4 and 20.2).  Its entries are either the initial object $\varnothing$ or a finite coproduct of the monoidal unit $\tensorunit$.  There is a unique operadic collection
\[\bigl(\{(1,1)\},\As\bigr)\]
such that for each map $f : \fC \to \fD \in \set$, the map
\[\nicexy{\Asc \ar[r]^-{\As_f} & \Asd \in \operadm}\]
is entrywise either uniquely determined by the identity map of the  monoidal unit in $\M$ or the unique map from $\varnothing$ to $\varnothing$ or $\tensorunit$.  The category $\alg\bigl(\{(1,1)\},\As\bigr)$ is exactly the category $\Catm$ of small $\M$-enriched categories \cite{borceux} (6.2). 
\end{example}

\begin{example}[$\M$-Enriched Categories with Extra Entries]\label{ex:enriched.cat.extra}
This example will play an important role in Sections \ref{sec:dk.enrichedcat.extra} and \ref{sec:alg.opcoll}.  Suppose $S \subseteq \bbNsq$ such that $(1,1) \in S$.  For each set $\fC$ define the $\sofc$-colored operad $\Ascs$ with entries
\[\Ascs\tzerotonetn = \begin{cases}
\Asc\tzerotonetn & \text{ if $|t_i| = (1,1)$ for all $i$},\\ 
\tensorunit & \text{ if $\tzerotonetn = \tzerotzero$},\\
\varnothing & \text{ otherwise}\end{cases}\]
where $n \geq 0$, each $t_i \in \sofc$, and $|t_i| = (|\uc|,|\ud|)$ if $t_i = \dc$.  The operad structure of $\Ascs$ is uniquely determined by that of $\Asc$ in Example \ref{ex:enriched.cat}.  There is a unique operadic collection
\[(S, \Ass)\]
such that for each map $f : \fC \to \fD \in \set$, the map
\[\nicexy{\Ascs \ar[r]^-{\Ass_f} & \Asds \in \operadm}\]
is entrywise either uniquely determined by the identity map of the monoidal unit in $\M$ or the unique map from $\varnothing$ to $\varnothing$ or $\tensorunit$.  Note that if $S = \{(1,1)\}$, then $(S,\Ass) = \bigl(\{(1,1)\},\As\bigr)$, the operadic collection for small $\M$-enriched categories in Example \ref{ex:enriched.cat}.

An $(S,\Ass)$-algebra is a pair $(\fC,A)$ consisting of a color set $\fC$ and an object $A \in \mtosofc$ such that:
\begin{itemize}
\item $\Acat \defn \bigl(\fC, \left\{A\dcsingle : c, d \in \fC\right\}\bigr)$ is an $\Asc$-algebra, i.e., an $\M$-enriched category with object set $\fC$;
\item There are no structure maps on the entries $A\dc$ if $(|\uc|,|\ud|) \not= (1,1)$.
\end{itemize}
We will call
\begin{itemize}
\item an $(S,\Ass)$-algebra an \emph{$\M$-enriched category with $S$-entries},
\item the $\Asc$-algebra $\Acat$ the \emph{categorical part} of $(\fC,A)$, and
\item the other entries the \emph{non-categorical part} of $(\fC,A)$. 
\end{itemize}
 According to our notational convention, in an entry $A\dcsingle \in \M$, $c$ is the input, and $d$ is the output.  So in usual categorical language, $A\dcsingle$ is the hom-object from $c$ to $d$.

Likewise, a map $f : (\fC,A) \to (\fD,B) \in \alg(S,\Ass)$ consists of
\begin{itemize}
\item a map $f : \fC \to \fD$ of color sets and
\item an entry map $f : A\dc \to B\fdc \in \M$ for each $\dc \in \sofc$ such that, when restricted to the categorical parts, these maps assemble to a functor $\fcat : \Acat \to \Bcat$ of $\M$-enriched categories.
\end{itemize}
There are no extra conditions imposed on the entry maps when restricted to the non-categorical parts.
\end{example}

\begin{example}[Operad-Like Structure]\label{ex:gprops}
The reference for this example is \cite{jy2}, to which we refer the reader for specific definitions.  For our current objective, the specific definitions of graphs are not that important.  What matters is that, in the following table, each category in the right-most column can be realized as the category of algebras of some operadic collection.  

Suppose $\G = (S,\tg)$ is a $1$-colored pasting scheme \cite{jy2} (Def. 8.2), with the corresponding $\fC$-colored version denoted $\Gc = (S,\tgc)$.  For each set $\fC$, according to \cite{jy2} (Lemma 14.4) there is an $\sofc$-colored operad $\sGc$ whose category of algebras is the category of $\Gc$-props in $\M$ (Def. 10.39).  Each entry of $\sGc$ is a coproduct of copies of the monoidal unit, indexed by the set of strict isomorphism classes of ordered graphs in $\Gc$ with the given vertex profiles and graph profiles.  There is a unique operadic collection
\[(S,\sG)\]
such that, for each map $f : \fC \to \fD \in \set$, the map
\[\nicexy{\sGc \ar[r]^-{\sG_f} & \sGd \in \operadm}\]
is induced by applying $f$ to the coloring of $\fC$-colored ordered graphs \cite{jy2} (Def. 1.20).  Here are some specific examples that extend those from \cite{jy2} (1.4.4, 10.5.3, and Ch. 11).
\begin{center}
\begin{small}
{\renewcommand{\arraystretch}{1.2} 
\begin{tabular}{|c|c|c|}\hline
$S \subseteq \bbNsq$ & $\tg$ & $\alg(S,\sG) =$ category of all small $\cdots$ in $\M$\\ \hline\hline
$\{(1,1)\}$ & unital linear graphs & enriched categories (Ex. \ref{ex:enriched.cat})\\ \hline
$\bbNsq$ & iterated graftings of corollas & vprops \cite{jy}\\ \hline 
$\bbN \times \{1\}$ & unital trees & colored operads (Def. \ref{def:all.operads})\\ \hline
$\bbNsq$ & simply-connected graphs & dioperads \cite{gan}\\ \hline
$\bbNsq$ & connected wheel-free graphs & properads\cite{vallette}\\ \hline
$\bbNsq$ & wheel-free graphs & props \cite{maclane63,maclane65}\\ \hline
$\bbN \times \{0,1\}$ & wheeled trees & wheeled operads \cite{mms}\\ \hline
$\bbNsq$ & connected wheeled graphs & wheeled properads \cite{mms}\\ \hline
$\bbNsq$ & wheeled graphs & wheeled props \cite{mms}\\ \hline
\end{tabular}}
\end{small}
\end{center}
\end{example}

\subsection{Maps of Operadic Collections}

\begin{definition}\label{def:maps.operad.coll}
Suppose $(S^1,\sO^1)$ and $(S^2,\sO^2)$ are two operadic collections in $\M$ (Def. \ref{def:operadic.collection}).  A \emph{map of operadic collections} $\alpha : (S^1,\sO^1) \to (S^2,\sO^2)$ consists of
\begin{enumerate}
\item an inclusion $S^1 \subseteq S^2 \subseteq \bbNsq$ and
\item a natural transformation
\[\nicexy{\alpha : \sO^1 \Longrightarrow \sO^2 : \set \ar[r] & \operadm}\]
such that, for each set $\fC$, the map $\alphac : \osubc^1 \to \osubc^2 \in \operadm$ on colors is given by the inclusion $S^1(\fC) \subseteq S^2(\fC)$.
\end{enumerate}
In other words, a typical entry map of $\alphac$ is a map
\[\nicexy{\osubc^1\tzerotonetn \ar[r]^-{\alphac} & \osubc^2\tzerotonetn \in \M}\]
with $n \geq 0$ and each $t_i \in S^1(\fC)$.
\end{definition}

\begin{example}\label{ex:changeS}
Suppose $S^1 \subseteq S^2 \subseteq \bbNsq$.
\begin{enumerate}
\item In the context of Example \ref{ex:initial.collection}, there is a map of operadic collections
\begin{equation}\label{initial.opcoll.s}
\nicexy{(S^1,\sI) \ar[r]^-{\iota} & (S^2,\sI)}
\end{equation}
such that, for each set $\fC$, the map $\iotac \in \operadm$ is entrywise the identity map.
\item In the context of Example \ref{ex:enriched.cat.extra}, there is a map of operadic collections
\begin{equation}\label{ass.onetwo}
\nicexy{\bigl(S^1,\As^{S^1}\bigr) \ar[r]^-{\iota} & \bigl(S^2,\As^{S^2}\bigr)}
\end{equation}
such that, for each set $\fC$, the map $\iotac \in \operadm$ is entrywise the identity map.
\end{enumerate}
\end{example}

\begin{example}\label{ex:opcoll.from.initial}
Suppose $(S,\sO)$ is an operadic collection.  Then there is a canonical map of operadic collections
\begin{equation}\label{opcoll.from.initial}
\nicexy{(S,\sI) \ar[r]^-{\eta} & (S,\sO)}
\end{equation}
such that, for each set $\fC$, the map $\etac : \sI_{\fC} \to \osubc$ is the unique map from the initial $\sofc$-colored operad $\sI_{\fC}$ to $\osubc$ in $\operadscm$.  In other words, for a fixed $S$, $(S,\sI)$ is the initial object among the operadic collections with the same $S$.
\end{example}

\begin{example}\label{ex:maps.op.coll}
The operadic collections in Section \ref{sec:ex.op.coll} are related as in the following commutative diagram of operadic collections.  In every case, it is assumed that $(1,1) \in S \subseteq \bbNsq$.  If $S$ is not explicitly specified, then it is as in the table in Example \ref{ex:gprops}.
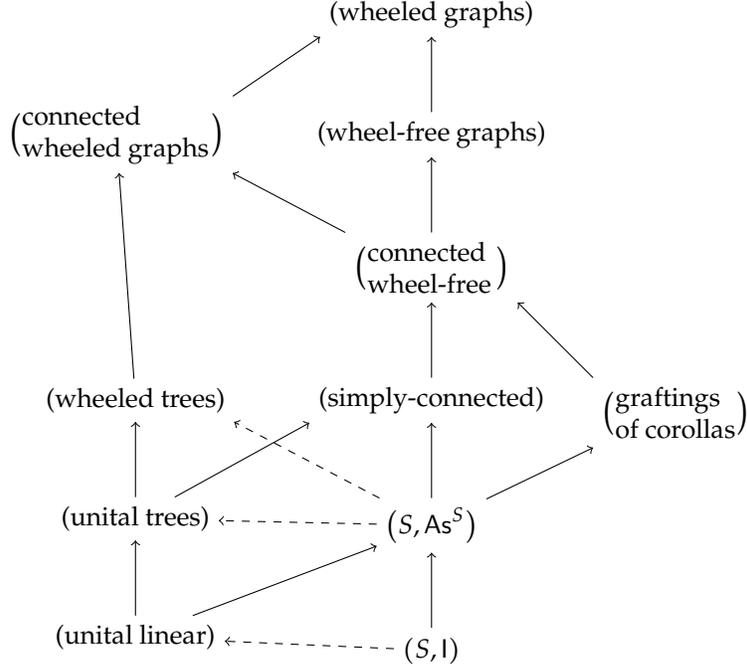
\begin{center}
\begin{small}
\begin{tikzpicture}
\node (whgr) {(wheeled graphs)};
\node[below=of whgr] (whfree) {(wheel-free graphs)};
\node[left=of whfree] (cwhgr) {$\Bigl($\parbox{2.5cm}{connected wheeled graphs}$\Bigr)$};
\node[below=of whfree] (cwhfree) {$\Bigl($\parbox{1.7cm}{connected wheel-free}$\Bigr)$};
\node[below=of cwhfree] (simply) {(simply-connected)};
\node[left=of simply] (whtree) {(wheeled trees)};
\node[below right=of cwhfree] (graftcor) {$\Bigl($\parbox{1.6cm}{graftings of corollas}$\Bigr)$};
\node[below=of whtree] (utree) {(unital trees)};
\node[below=of simply] (ass) {$\bigl(S,\Ass\bigr)$};
\node[below=of utree] (ulinear) {(unital linear)}; 
\node[below=of ass] (initial) {$(S,\sI)$};
\draw[->] (cwhgr.north east)--(whgr.south west); \draw[->] (whfree)--(whgr); 
\draw[->] (cwhfree.north west)--(cwhgr.south east);
\draw[->] (cwhfree)--(whfree); \draw[->] (whtree)--(cwhgr);
\draw[->] (graftcor.north west)--(cwhfree.south east); \draw[->] (ass)--(graftcor.south west);
\draw[->] (simply)--(cwhfree); \draw[->] (utree)--(simply.south west); 
\draw[->] (utree)--(whtree);
\draw[->] (ass)--(simply); \draw[dashed,->] (ass)--(utree); 
\draw[dashed,->] (ass)--(whtree.south east);
\draw[->] (initial)--(ass); \draw[->] (ulinear)--(utree); \draw[->] (ulinear)--(ass);
\draw[dashed,->] (initial)--(ulinear);
\end{tikzpicture}
\end{small}
\end{center}
From $(S,\Ass)$ the dashed map to (unital trees) is defined if $S \subseteq \bbN \times \{1\}$, and the dashed map from $(S,\Ass)$ to (wheeled trees) is defined if $S \subseteq \bbN \times \{0,1\}$.  The dashed map from $(S,\sI)$ to (unital linear) is defined if $S = \{(1,1)\}$.
\end{example}

\begin{definition}\label{def:op.coll.map.rightadj}
Suppose $\alpha : (S^1,\sO^1) \to (S^2,\sO^2)$ is a  map of operadic collections in $\M$.  Define the \emph{restriction functor}
\begin{equation}\label{restriction.functor}
\nicexy{\algsoone & \algsotwo \ar[l]_-{\alpha^\ast}}
\end{equation}
as follows.  For $(\fC,A) \in \algsotwo$ define
\[\alpha^*(\fC,A) = \bigl(\fC,\alphac^*A\bigr) \in \algsoone\]
in which the \emph{local restriction functor}
\begin{equation}\label{restriction.c}
\nicexy{\alg(\osubcone) & \alg(\osubctwo) \ar[l]_-{\alphac^\ast}}
\end{equation}
is induced by the map $\alphac : \osubcone \to \osubctwo \in \operadm$ as in Def. \ref{def:operadic.collection.algebra}(2).
\end{definition}

\begin{proposition}\label{opcoll.map.leftadj}
Suppose $\alpha : (S^1,\sO^1) \to (S^2,\sO^2)$ is a  map of operadic collections in a bicomplete symmetric monoidal closed category $\M$.  Then there is an adjunction
\[\nicexy{\algsoone \ar@<2pt>[r]^-{\alpha_!} & \algsotwo \ar@<2pt>[l]^-{\alpha^\ast}}\]
in which the right adjoint $\alpha^*$ is the restriction functor \eqref{restriction.functor}
\end{proposition}

\begin{proof}
There is a commutative diagram
\[\nicexy{(S^1,\sO^1) \ar[r]^-{\alpha} & (S^2,\sO^2)\\
(S^1,\sI) \ar[u]^-{\eta} \ar[r]^-{\iota} & (S^2,\sI) \ar[u]_-{\eta}}\]
of operadic collections, in which the bottom map is the one in \eqref{initial.opcoll.s} and the vertical maps are as in \eqref{opcoll.from.initial}.  There is an induced commutative diagram of restriction functors
\begin{equation}\label{opcoll.adjoint.lifting.diagram}
\nicexy{\alg(S^1,\sO^1) \ar@<2pt>@{.>}[r]^-{\alpha_!} \ar@<2pt>[d]^-{\eta^\ast}
& \alg(S^2,\sO^2) \ar@<2pt>[l]^-{\alpha^\ast} \ar@<2pt>[d]^-{\eta^\ast}\\
\alg(S^1,\sI) \ar@<2pt>[u]^-{\eta_!} \ar@<2pt>[r]^-{\iota_!} 
& \alg(S^2,\sI) \ar@<2pt>[u]^-{\eta_!} \ar@<2pt>[l]^-{\iota^\ast}}
\end{equation}
together with adjunctions $(\eta_!,\eta^*)$ and $(\iota_!,\iota^*)$.  Indeed, the right adjoint $\iota^*$ forgets some entries, and $\iota_!$ adds some $\varnothing$ entries.  Each right adjoint $\eta^*$ forgets about the operadic algebra structure maps, while
\[\eta_!(\fC,A) = \bigl(\fC,\osubc^i \comp A\bigr)\]
is the free $\osubc^i$-algebra functor (Def. \ref{colored-operad-algebra}) for $(\fC,A) \in \alg(S^i,\sI)$ for both $i=1,2$.  See Example \ref{ex:initial.collection} for an explicit description of each $\alg(S^i,\sI)$.  Moreover, both vertical right adjoints $\eta^*$ are monadic \cite{borceux} (4.4.1), and $\algsotwo$ is cocomplete (Proposition \ref{algso.bicomplete}).  Since $\iota^*$ admits a left adjoint, the Adjoint Lifting Theorem \cite{borceux} (4.5.6) implies that $\alpha^*$ also admits a left adjoint $\alpha_!$.
\end{proof}

\begin{remark}\label{rk:opcoll.leftadj.diagram}
In the diagram \eqref{opcoll.adjoint.lifting.diagram}, the right adjoint diagram is commutative, i.e.,
\[\eta^* \alpha^* = \iota^* \eta^*.\]
By uniqueness the left adjoint diagram is also commutative up to natural isomorphism, i.e.,
\[\alpha_! \eta_! \cong \eta_! \iota_!.\]
\end{remark}

\begin{example}\label{ex:algebras.op.coll}
Consider the commutative diagram of operadic collections in Example \ref{ex:maps.op.coll}.  By Proposition \ref{opcoll.map.leftadj} there is a commutative diagram of categories of algebras over operadic collections, in which each map is a left adjoint to some forgetful functor:
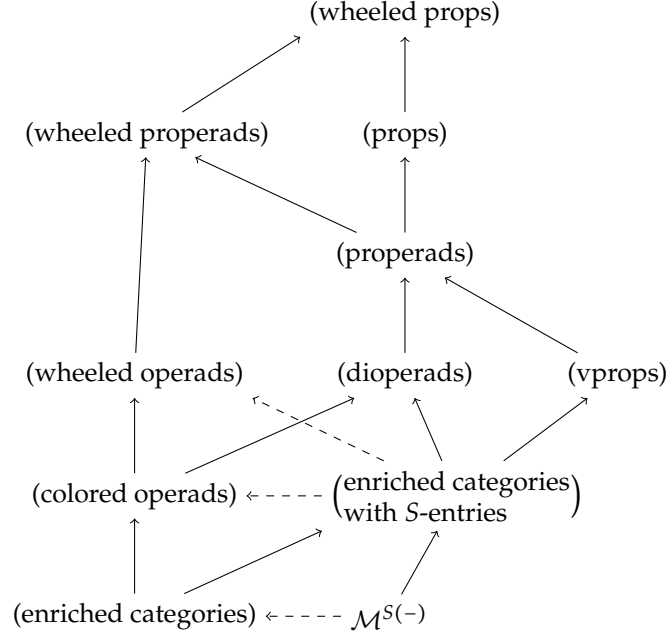
\begin{center}
\begin{small}
\begin{tikzpicture}
\node (whgr) {(wheeled props)};
\node[below=of whgr] (whfree) {(props)};
\node[left=of whfree] (cwhgr) {(wheeled properads)};
\node[below=of whfree] (cwhfree) {(properads)};
\node[below=of cwhfree] (simply) {(dioperads)};
\node[left=of simply] (whtree) {(wheeled operads)};
\node[below right=of cwhfree] (graftcor) {(vprops)};
\node[below=of whtree] (utree) {(colored operads)};
\node[right=of utree] (ass) {$\Bigl($\parbox{3cm}{enriched categories with $S$-entries}$\Bigr)$};
\node[below=of utree] (ulinear) {(enriched categories)}; 
\node[right=of ulinear] (initial) {$\mtosdash$};
\draw[->] (cwhgr)--(whgr.south west); \draw[->] (whfree)--(whgr); 
\draw[->] (cwhfree)--(cwhgr);
\draw[->] (cwhfree)--(whfree); \draw[->] (whtree)--(cwhgr);
\draw[->] (graftcor)--(cwhfree); \draw[->] (ass)--(graftcor);
\draw[->] (simply)--(cwhfree); \draw[->] (utree)--(simply); \draw[->] (utree)--(whtree); 
\draw[->] (ass)--(simply); \draw[dashed,->] (ass)--(utree); 
\draw[dashed,->] (ass)--(whtree.south east);
\draw[->] (initial)--(ass); \draw[->] (ulinear)--(utree); \draw[->] (ulinear)--(ass.south west);
\draw[dashed,->] (initial)--(ulinear);
\end{tikzpicture}
\end{small}
\end{center}
In each case, the category is enriched in the underlying category $\M$, and the dashed arrows were explained in the aforementioned example.  For those categories that arise from pasting schemes as in Example \ref{ex:gprops}, the left adjoint, such as (props) $\to$ (wheeled props), on each object is described explicitly in \cite{jy2} (section 12.1.3).

Furthermore, in the adjunction
\begin{equation}\label{enriched.cat.adjoint}
\nicexy@R-.5cm{(\text{enriched categories}) \ar@<2pt>[r]^-{\text{add $\varnothing$}} &
(\text{enriched categories with $S$-entries}) \ar@<2pt>[l]^-{(-)^{\cat}}\\
\Catm \ar@{=}[u] & \alg(S,\Ass) \ar@{=}[u]}
\end{equation}
the right adjoint $(-)^{\cat}$ simply forgets about the non-categorical part as defined in Example \ref{ex:enriched.cat.extra}.  The left adjoint adds $\varnothing$ entries for the non-categorical part.
\end{example}

\section{Homotopy Theory of Enriched Categories}
\label{sec:dk.enriched.cat}

In this section we recall the Dwyer-Kan model category structure on the category of all small $\M$-enriched categories.  Assume that the underlying category is a  bicomplete symmetric monoidal closed category $(\M, \otimes, \tensorunit)$ with initial object $\varnothing$.

\subsection{Monoidal Model Categories}

For the reader's convenience, let us briefly recall some key concepts in model category theory.  Our main references here are \cite{hirschhorn,hovey,ss,ss03}.  

A model category is \emph{right proper} if the pullback of a weak equivalence along a fibration is again a weak equivalence.

A model category is \emph{cofibrantly generated} \cite{hovey} (2.1.17) if it is equipped with a set $I$ of cofibrations and a set $J$ of trivial cofibrations (i.e., maps which are both cofibrations and weak equivalences) that permit the small object argument (with respect to some cardinal $\kappa$), and a map is a (trivial) fibration if and only if it satisfies the right lifting property with respect to all maps in $J$ (resp., $I$). 

An object is \emph{small} if there is some regular cardinal $\kappa$ for which it is $\kappa$-small.  A \emph{strongly cofibrantly generated model category} is a cofibrantly generated model category in which the domains of the maps in the sets $I$ and $J$ are small.  A \emph{combinatorial model category} is a cofibrantly generated model category whose underlying category is locally presentable \cite{ar94}.

\begin{definition}\label{def:ppaxiom}
A symmetric monoidal closed category $\calm$ equipped with a model structure is called a \emph{monoidal model category} \cite{ss} if it satisfies the following \emph{pushout product axiom}: 
\begin{quote}
Given any cofibrations $f:X_0\to X_1$ and $g:Y_0\to Y_1$, the pushout corner map
\[\nicexy{X_0\otimes Y_1 \coprod\limits_{X_0\otimes Y_0}X_1\otimes Y_0 
\ar[r]^-{f\boxprod g} & X_1\otimes Y_1}\]
is a cofibration that is also a weak equivalence if either $f$ or $g$ is also a weak equivalence.
\end{quote}
By adjunction the pushout product axiom can also be stated in terms of the internal hom; see \cite{hovey} (4.2.2).  A monoidal model category is said to satisfy the \emph{monoid axiom} if every transfinite composition of pushouts of tensor products of trivial cofibrations with an arbitrary object is a weak equivalence.
\end{definition}

\subsection{Dwyer-Kan Model Structure of Enriched Categories}

Next we recall the model structure on $\M$-enriched categories from \cite{bm13,muro15}.  The following concept is needed to define fibrations of $\M$-enriched categories.

\begin{definition}\label{def:isofibration}
A functor $F : A \to B$ between two categories is an \emph{isofibration} if for each object $a \in A$ and each isomorphism $f : Fa \to b$ in $B$, there exists an isomorphism $g : a \to a'$ in $A$ such that $Fa' = b$ and $Fg = f$.   
\end{definition}

Isofibrations are the fibrations in the folk model structure of $\Cat$ \cite{jt91,rezk}, the category of all small categories and functors between them.

\begin{definition}\label{def:pizero}
Suppose $\M$ is a monoidal model category and $\Catm$ is the category of small categories enriched in $\M$ and functors between them \cite{borceux} (6.2).  
\begin{enumerate}
\item There is a functor
\begin{equation}\label{functor.pizero}
\nicexy{\Catm \ar[r]^-{\pi_0} & \Cat}
\end{equation}
that sends an $\M$-enriched category $A$ to the category $\pi_0(A)$ with the same objects as $A$ and, for objects $x$ and $y$ in $A$, hom set
\[\pi_0(A)(x,y) = \Ho(\M)\bigl(\tensorunit, A(x,y)\bigr).\]
Here $\Ho(\M)$ is the homotopy category of $\M$, and $\tensorunit$ is the monoidal unit in $\M$.  We will write $\pizerom$ for $\pizero$ if we need to emphasize $\M$.
\item A map $f : A \to B \in \Catm$ is \emph{homotopically essentially surjective} if $\pi_0(f) \in \Cat$ is an essentially surjective functor; i.e., each object in $B$ is isomorphic to an object of the form $\pizero(f)(a)$ for some object $a \in A$.
\item A map $f \in \Catm$ is called a \emph{homotopical isofibration} if $\pi_0(f) \in \Cat$ is an isofibration.
\end{enumerate}
\end{definition}

\begin{definition}\label{def:dk-weakeq-fib}
Suppose $\M$ is a monoidal model category and $f : A \to B \in \Catm$ is a functor of small $\M$-enriched categories.
\begin{enumerate}
\item We call $A$ \emph{locally fibrant} if for each pair of objects $x,y \in A$, the object $A(x,y) \in \M$ is fibrant.
\item We call $f$ a \emph{local weak equivalence} (resp., \emph{local (trivial) fibration}) if for each pair of objects $x,y \in A$, the map $f : A(x,y) \to B(x,y) \in \M$ is a weak equivalence (resp., (trivial) fibration).
\item We call $f$ a \emph{weak equivalence} if it is a local weak equivalence that is also  homotopically essentially surjective.
\item We call $f$ a \emph{fibration} if it is a local fibration that is also a homotopical isofibration.
\item If $\Catm$ admits the model category structure with the above weak equivalences and fibrations, then it is called the \emph{Dwyer-Kan model category structure}.
\end{enumerate}
\end{definition}

The following existence theorem is a combination of \cite{bm13} (1.10) and \cite{muro15} (1.1).  Our assumptions are strong enough to ensure that their model structures coincide.

\begin{theorem}\label{dk.catm}
Suppose $\M$ is a combinatorial monoidal model category that satisfies the following conditions:
\begin{itemize}
\item $\M$ satisfies the monoid axiom, has a cofibrant monoidal unit, and is right proper.
\item The class of weak equivalences is closed under filtered colimits.  
\end{itemize}
Then $\Catm$ admits the Dwyer-Kan model category structure and is a combinatorial model category.  Moreover, in $\Catm$:
\begin{itemize}
\item A map is a trivial fibration if and only if it is a local trivial fibration that is also surjective on objects.  
\item An object is fibrant if and only if it is locally fibrant.  
\item The set of generating (trivial) cofibrations is $I'$ in \eqref{Iprime} (resp., $J'$ in \eqref{Jprime}) below.
\end{itemize}
\end{theorem}

\subsection{Generating (Trivial) Cofibrations}\label{sec:catm.generators}

We now define the sets $I'$ and $J'$ of generating cofibrations and generating trivial cofibrations in $\Catm$  in Theorem \ref{dk.catm}.

\begin{definition}
For a set $\fC$, denote by $\Catmc$ the category of $\M$-enriched categories with object set $\fC$ and object-preserving functors.
\end{definition}

\begin{definition}
Suppose $\M$ is a cofibrantly generated monoidal model category with $I$ as its set of generating cofibrations.
\begin{enumerate}
\item Denote by $\tone \in \Catm^{\{a\}}$ the $\M$-enriched category with only one object $a$ such that $\tone(a,a) = \tensorunit$, the monoidal unit in $\M$.
\item For a $2$-element set $\{a,b\}$, we write $C_{1,1}$ for the left adjoint in the adjunction
\begin{equation}\label{coneone}
\nicexy{\M \ar@<2pt>[r]^-{C_{1,1}} & \Catmab \ar@<2pt>[l]^-{U_{1,1}}}
\end{equation}
in which
\[U_{1,1}(Y) = Y(a,b)\]
for $Y \in \Catmab$.
\item Define
\begin{equation}\label{Iprime}
I'  = C_{1,1}(I) \coprod \{\varnothing \to \tone\} \subseteq \Catm.
\end{equation}
Here $\varnothing$ denotes the $\M$-enriched category with an empty set of object, and $C_{1,1}(I)$ is regarded as a set of maps in $\Catm$.
\end{enumerate}
\end{definition}

\begin{remark}\label{rk:coneone}
For an object $X \in \M$, the $\M$-enriched category $C_{1,1}X \in \Catmab$ \eqref{coneone} has hom objects
\[\bigl(C_{1,1}X\bigr)(x,y) = \begin{cases} X & \text{ if $(x,y) = (a,b)$},\\
\tensorunit & \text{ if $(x,y) = (a,a)$ or $(b,b)$},\\
\varnothing & \text{ if $(x,y) = (b,a)$}.\end{cases}\]
\end{remark}

Next we define the set $J'$ of generating trivial cofibrations in $\Catm$, for which we first need some definitions.

\begin{definition}\label{def:comon.interval}
Suppose $\M$ is a monoidal model category.  A \emph{(cocommutative) comonoidal interval} in $\M$ \cite{bm03,bm07} is a factorization
\begin{equation}\label{comonoidal.interval}
\nicexy{\tensorunit \coprod \tensorunit \ar@{>->}[r]^-{\epsilon} \ar `u[rr] `[rr]|-{\text{fold}} [rr]
& H \ar[r]^-{\beta}_-{\sim} & \tensorunit}
\end{equation}
of the fold map of the monoidal unit $\tensorunit$ in which $H$ is a counital coassociative (cocommutative) comonoid in $\M$, both $\epsilon$ and $\beta$ are maps of comonoids, $\epsilon$ is a cofibration, and $\beta$ is a weak equivalence.
\end{definition}

\begin{definition}\label{def:monoidal.functor}
A \emph{lax monoidal functor} $F : \M \to \N$ between two monoidal categories is a functor equipped with structure maps
\[\nicexy{FX \otimes FY \ar[r]^-{F^2_{X,Y}} & F(X \otimes Y), \quad \tensorunit^{\N} \ar[r]^-{F^0} & F\tensorunit^{\M}}\]
for $X$ and $Y$ in $\M$ that are associative and unital in a suitable sense \cite{maclane} (XI.2).  If, furthermore, $\M$ and $\N$ are symmetric monoidal categories, and $F^2$ is compatible with the symmetry isomorphisms, then $F$ is called a \emph{lax symmetric monoidal functor}.
\end{definition}

Note that what is called a lax monoidal functor here is simply called a monoidal functor in \cite{maclane}.  

\begin{definition}\label{def:fibrant.replacement}
Suppose $\calc$ is a category equipped with a subclass of objects called \emph{fibrant objects} and a subclass of maps called \emph{weak equivalences}.
\begin{enumerate}
\item A \emph{fibrant replacement functor} in $\calc$ is a pair $(R, \eta)$ consisting of
\begin{itemize}
\item a functor $R : \calc \to \calc$ and 
\item a natural transformation $\eta : \Id \to R$ 
\end{itemize}
such that, for each object $X$ in $\calc$, $RX$ is a fibrant object and $\eta_X : X \to RX$ is a weak equivalence.
\item Suppose further that $\calc$ is a symmetric monoidal category.  A fibrant replacement functor $(R,\eta)$ in $\calc$ is said to be \emph{lax symmetric monoidal} if $R$ is a lax symmetric monoidal functor such that $\eta$ is a monoidal natural transformation \cite{maclane} (XI.2).
\end{enumerate}
\end{definition}

The following model structure is needed to define the set of generating trivial cofibrations in $\Catm$.

\begin{proposition}\label{catmc-model}
Suppose $\M$ is a combinatorial monoidal model category that satisfies \underline{either one} of the following two conditions.
\begin{enumerate}
\item $\M$ satisfies the monoid axiom. 
\item $\M$ has a cofibrant monoidal unit, a lax symmetric monoidal fibrant replacement functor, and a comonoidal interval.
\end{enumerate}
Then for each set $\fC$, the category $\Catmc$ admits a combinatorial model category structure in which weak equivalences and fibrations are defined entrywise in $\M$.
\end{proposition}

We will call this the \emph{projective model category structure} on $\Catmc$.  The first case of the previous proposition is a special case of the colored version of \cite{ss} 4.1(3).  The second case is a special case of \cite{bm07} (2.1). 

\begin{example}
Examples of categories that satisfy Prop. \ref{catmc-model}(2) include the categories of simplicial sets, of chain complexes over a characteristic zero field $k$, of simplicial $k$-modules \cite{quillen.homotopical}, and of small categories with the folk model structure \cite{jt91,rezk}.  Categories that satisfy Prop. \ref{catmc-model}(1) (i.e., the monoid axiom) include the ones in the previous sentence,  the categories of symmetric spectra \cite{hovey-shipley-smith} and of $\Gamma$-spaces \cite{bf,lydakis,schwede,segal}, and the stable module category of left $kG$-modules for a field $k$ and a finite group $G$; see \cite{hovey} (2.2) and \cite{ss} (Section 5).
\end{example}

\begin{definition}
Suppose $\M$ is a cofibrantly generated monoidal model category with $J$ as its set of generating trivial cofibrations such that $\Catmab$ admits the projective model category structure in Prop. \ref{catmc-model}.
\begin{enumerate}
\item Define $\bbI \in \Catmab$ whose four hom objects are all equal to the monoidal unit  $\tensorunit$.
\item An \emph{$\M$-interval} \cite{bm13} (1.11) is a cofibrant object in $\Catmab$ that is weakly equivalent to $\bbI$.  A set $\cJ$ of $\M$-intervals is \emph{generating} if each $\M$-interval is a retract of a trivial extension of some $\M$-interval in $\cJ$.
\item Suppose $\M$ has a generating set of $\M$-intervals $\cJ$--e.g., $\M$ is combinatorial \cite{bm13} (1.12).  Define
\begin{equation}\label{Jprime}
J' = C_{1,1}(J) \coprod \bigl\{\tone \to T : T \in \cJ \bigr\} \subseteq \Catm
\end{equation}
in which each map $\tone \to T$ sends the object $a \in \tone$ to the object $a \in T$.
\end{enumerate}
\end{definition}

\begin{remark}\label{rk:catm.cofibration}
Using the projective model category structure in Prop.\ref{catmc-model}, it is possible to describe the cofibrations in the Dwyer-Kan model category structure on $\Catm$ (Theorem \ref{dk.catm}). If $f : \fC \to \fD$ is a map of sets, then there is an induced change-of-color adjunction
\[\nicexy{\Catmc \ar@<2pt>[r]^-{f_!} & \Catmd \ar@<2pt>[l]^-{f^*}}\]
whose right adjoint is defined by $(f^*B)(x,y) = B(fx,fy)$ for $B \in \Catmd$ and $x,y \in \fC$.  Then a cofibration $\alpha : A \to B \in \Catm$ in the Dwyer-Kan model category structure is exactly a functor such that the following two conditions hold.
\begin{enumerate}
\item On objects $\alpha : \Ob(A) \to \Ob(B)$ is injective.
\item The induced functor $\alpha_! A \to B \in \Catm^{\Ob(B)}$ is a cofibration in the projective model category structure in Prop.\ref{catmc-model}. 
\end{enumerate}
See \cite{bm13} (Remark 1.7).
\end{remark}

\section{Homotopy Theory of Enriched Categories with Extra Entries}
\label{sec:dk.enrichedcat.extra}

The purpose of this section is to prove Theorem \ref{algass.dk.model}, in which we show that, for a suitable model category $\M$, the category $\alg(S,\Ass)$ of $\M$-enriched categories with $S$-entries (Example \ref{ex:enriched.cat.extra}) admits a cofibrantly generated Dwyer-Kan model structure that extends the one on $\M$-enriched categories (Theorem \ref{dk.catm}).  This is an intermediate step toward showing that the category $\algso$ for a suitable operadic collection $(S,\sO)$ admits a Dwyer-Kan model structure.  This will include all the categories in the diagram in Example \ref{ex:algebras.op.coll} that receive a (dashed) directed path from $\alg(S,\Ass)$, i.e., all but $\mtosdash$ and enriched categories.

\subsection{Generating (Trivial) Cofibrations}

We first define the sets of generating (trivial) cofibrations.  The following symbols will be used to parametrized inputs and outputs.

\begin{notation}
Fix once and for all two sequences of independent symbols $\{a_i\}_{i \geq 1}$ and $\{b_i\}_{i \geq 1}$.    For $m \geq 1$ we will write $a_{[1,m]}$ for the finite set (or the finite sequence) $\{a_1,\ldots,a_m\}$, and likewise for $b_{[1,m]}$.  In the context of Section \ref{sec:catm.generators}, we will regard $a=a_1$ and $b=b_1$.
\end{notation}

\begin{definition}\label{def:cmn}
Suppose $\M$ is a combinatorial monoidal model category with generating (trivial) cofibrations $I$ (resp., $J$) such that $\Catmab$ admits the projective model category structure in Prop. \ref{catmc-model}.  Suppose $S \subseteq \bbNsq$ with $(1,1) \in S$.  Consider the category $\alg(S,\Ass)$ of $\M$-enriched categories with $S$-entries (Example \ref{ex:enriched.cat.extra}), which we will abbreviate to $\algass$.
\begin{enumerate}
\item For $(m,n) \in S$ define a functor
\[\nicexy{\M \ar[r]^-{C_{m,n}} & \algass}\]
by the universal property that there is a natural isomorphism
\begin{equation}\label{cmn.functor}
\algass\bigl(C_{m,n}X, (\fD,B)\bigr) \cong \coprod_{\dc \in \sofd,\, (|\uc|,|\ud|) = (m,n)} \M\left(X, B\dc\right)
\end{equation}
for all $X \in \M$ and $(\fD,B) \in \algass$.
\item Define the sets of maps
\begin{equation}\label{IsJs}
\begin{split}
I^S &= I' \amalg \coprod_{(m,n) \in S \setminus \{(1,1)\}} C_{m,n}I \\
J^S &= J' \amalg \coprod_{(m,n) \in S \setminus \{(1,1)\}} C_{m,n}J 
\end{split}
\end{equation}
in $\algass$.  Here $I' \subseteq \Catm$ \eqref{Iprime} and $J' \subseteq \Catm$ \eqref{Jprime} are regarded as in $\algass$ via the left adjoint in \eqref{enriched.cat.adjoint}, which simply adds $\varnothing$ entries for the non-categorical part.
\end{enumerate}
\end{definition}

\begin{remark}\label{rk:cmn}
The functor $C_{m,n}$ \eqref{cmn.functor} may be constructed as follows.  For an object $X \in \M$, the object $C_{m,n}X \in \algass$ has color set $\left\{a_{[1,m]}, b_{[1,n]}\right\}$ and entries
\[(C_{m,n}X)\dc = \begin{cases}
X & \text{ if $\dc = \bonebnaoneam$},\\
\tensorunit & \text{ if $\dc = \aiai$ for some $1 \leq i \leq m$ or $\bjbj$ for some $1 \leq j \leq n$},\\
\varnothing & \text{ otherwise}.
\end{cases}\]
In view of Remark \ref{rk:coneone}, there is no conflict between the definition of $C_{m,n}$ in \eqref{cmn.functor} when $(m,n) = (1,1)$ and that of $C_{1,1}$ in \eqref{coneone}.
\end{remark}

\subsection{Dwyer-Kan Model Structure}

Recall the right adjoint in \eqref{enriched.cat.adjoint}
\[\nicexy{\Catm & \algass \ar[l]_-{(-)^{\cat}}}\]
that forgets about the non-categorical part.

\begin{definition}\label{def:algass.model}
Suppose $S \subseteq \bbNsq$ with $(1,1) \in S$.  Suppose $\M$ is a monoidal model category and $f :(\fC, A) \to (\fD,B) \in \algass$ is a map.
\begin{enumerate}
\item We call $A$ \emph{locally fibrant} if the object $A\dc \in \M$ is fibrant for each $\dc \in \sofc$.
\item We call $f$ a \emph{local weak equivalence} (resp., \emph{local (trivial) fibration}) if  the map $f : A\dc \to B\fdc \in \M$ is a weak equivalence (resp., (trivial) fibration) for each $\dc \in \sofc$.
\item We call $f$ a \emph{weak equivalence} if 
\begin{itemize}
\item it is a local weak equivalence and 
\item $\fcat \in \Catm$ is a weak equivalence in the sense of Def. \ref{def:dk-weakeq-fib}.  
\end{itemize}
Denote the class of weak equivalences in $\algass$ by $W^S$.
\item We call $f$ a \emph{fibration} if 
\begin{itemize}
\item it is a local fibration and 
\item $\fcat \in \Catm$ is a fibration in the sense of Def. \ref{def:dk-weakeq-fib}.
\end{itemize}
\item If $\algass$ admits the model category structure with the above weak equivalences and fibrations, then it is called the \emph{Dwyer-Kan model category structure}.
\end{enumerate}
\end{definition}

\begin{remark}\label{rk:algass.weakeq}
In the context of Definition \ref{def:algass.model}:
\begin{enumerate}
\item $f$ is a weak equivalence if and only if (i) it is a local weak equivalence and (ii) $\pizero(\fcat) \in \Cat$ is an equivalence of categories, i.e., fully faithful and essentially surjective \cite{maclane} (IV.4 Theorem 1).  Since condition (i) guarantees that $\pizero(\fcat)$ is fully faithful, condition (ii) may be replaced by (ii)': $\pizero(\fcat)$ is essentially surjective.
\item $f$ is a fibration if and only if (i) it is a local fibration and (ii) $\pizero(\fcat) \in \Cat$ is an isofibration.
\end{enumerate}
\end{remark}

\begin{notation}
For a class of maps $K$ in a category,  recall the notation of \cite{hovey} (2.1):
\begin{itemize}
\item $K$-cell is the class of maps obtained as transfinite compositions of pushouts of maps in $K$.
\item $K$-inj is the class of maps with the right lifting property with respect to every map in $K$.
\item $K$-cof is the class of maps with the left lifting property with respect to every map in $K$-inj.
\end{itemize}
\end{notation}

\begin{lemma}\label{jcell.in.w}
Suppose $\M$ is as in Theorem \ref{dk.catm} and $S \subseteq \bbNsq$ with $(1,1) \in S$.   Then there is an inclusion
\[J^S\text{-cell} \subseteq W^S.\]
\end{lemma}

\begin{proof}
First observe that by Remark \ref{rk:algass.weakeq} $W^S$ is closed under transfinite compositions because the class of weak equivalences in $\M$ and the class of equivalences of categories both have this property.  So it is enough to show that the pushout of each map in $J^S$ is a weak equivalence.  Consider a pushout
\[\nicexy{A \ar[d]_-{g} \ar[r]^-{i} & B \ar[d]^-{h}\\ C \ar[r] & D}\]
in $\algass$ with $g \in J^S$.  First suppose $g \in J' \subseteq J^S$ \eqref{Jprime}.  Then the restriction to the categorical part
\[\nicexy{\Acat \ar[d]_-{\gcat} \ar[r] & \Bcat \ar[d]^-{\hcat}\\ \Ccat \ar[r] & \Dcat}\]
is a pushout in $\Catm$.  Since $\gcat \in J'$, by \cite{muro15} (10.2) the map $\hcat \in \Catm$ is a weak equivalence.  Furthermore, the non-categorical part of $h$ is entrywise the identity map because the left adjoint from $\M$-enriched categories to $\algass$ \eqref{enriched.cat.adjoint} simply adds $\varnothing$ entries.  Therefore, we conclude that $h$ is a weak equivalence.

Next suppose $g \in C_{m,n}J$ for some $(m,n) \in S \setminus \{(1,1)\}$, say $g = C_{m,n}j$ for some map $j : A' \to C'$ in $J$.  By the defining property \eqref{cmn.functor} of $C_{m,n}$, if $\dc \in \sofc$ and $\dc \not= \bonebnaoneam$, then the corresponding entry of the map $h$ is the identity map.  In particular, the restriction $\hcat \in \Catm$ is a weak equivalence.  Furthermore, the square
\[\nicexy{A' \ar[d]_-{j} \ar[r]^-{i} & B\ibonebnaoneam \ar[d]^-{h}\\ C' \ar[r] & D\hibonebnaoneam}\]
is a pushout in $\M$.  Since $j \in \M$ is a trivial cofibration, this entry of $h$ is also a trivial cofibration.  We conclude that $h$ is a weak equivalence in $\algass$.  This proves that $J^S$-cell $\subseteq W^S$.
\end{proof}

\begin{lemma}\label{iinj.jinj.w}
Suppose $\M$ is as in Theorem \ref{dk.catm} and $S \subseteq \bbNsq$ with $(1,1) \in S$.   Then
\[J^S\text{-inj} \cap W^S = I^S\text{-inj}.\]
\end{lemma}

\begin{proof}
Suppose $h$ is a map in $\algass$.  By definition \eqref{IsJs} $h \in I^S$-inj if and only if:
\begin{enumerate}
\item $\hcat \in I'$-inj and
\item the non-categorical part of $h$ is entrywise a trivial fibration.
\end{enumerate}
Likewise, $h \in J^S$-inj if and only if:
\begin{enumerate}
\item $\hcat \in J'$-inj and
\item the non-categorical part of $h$ is entrywise a fibration.
\end{enumerate}
By \cite{muro15} (4.16)
\[J'\text{-inj} \cap W = I'\text{-inj},\]
where $W$ denotes the class of weak equivalences in $\Catm$.  Together with the above descriptions of $I^S$-inj and $J^S$-inj, this proves the desired equality.
\end{proof}

\begin{theorem}\label{algass.dk.model}
Suppose $\M$ is as in Theorem \ref{dk.catm} and $S \subseteq \bbNsq$ with $(1,1) \in S$.   Then $\algass$ admits the Dwyer-Kan model category structure (Def. \ref{def:algass.model}).  Moreover:
\begin{itemize}
\item This model category structure of $\algass$ is strongly cofibrantly generated with generating (trivial) cofibrations $I^S$ (resp., $J^S$) as in \eqref{IsJs}.  
\item A map in $\algass$ is a trivial fibration if and only if it is a local trivial fibration that is also surjective on colors.  
\item An object in $\algass$ is fibrant if and only if it is locally fibrant.
\end{itemize}
\end{theorem}

\begin{proof}
To see that the domain of each map in  $I^S \cup J^S$ is small, use Theorem \ref{dk.catm}, the fact that $\M$ is combinatorial, and the characterization \eqref{cmn.functor} of $C_{m,n}$.  The descriptions of trivial fibrations and fibrant objects follow from the Dwyer-Kan model category structure on $\Catm$ and Def. \ref{def:algass.model}.

To see that $\algass$ has the desired model category structure, we check the conditions in the recognition theorem \cite{hovey} (2.1.19) for cofibrantly generated model categories.  First, to see that $\algass$ is bicomplete, we use Prop. \ref{algso.bicomplete} and the fact that $\algass$ is the category of algebras of an operadic collection (Example \ref{ex:enriched.cat.extra}).

To see that the class of weak equivalences in $\algass$ has the $2$-out-of-$3$ property, observe that local weak equivalences in $\algass$ and weak equivalences in $\Catm$ all have the $2$-out-of-$3$ property.  Likewise, the class of weak equivalences in $\algass$ is closed under retracts.

Since $J^S$-cell $\subseteq J^S$-cof by \cite{hovey} (2.1.10), to use the recognition theorem \cite{hovey} (2.1.19) on $\algass$, it remains to show that:
\begin{enumerate}
\item $J^S$-cell $\subseteq W^S$.
\item $J^S$-inj $\cap W^S = I^S$-inj.
\end{enumerate}
Statement (1) holds by Lemma \ref{jcell.in.w}, and statement (2) holds by Lemma \ref{iinj.jinj.w}.
\end{proof}

\begin{remark}\label{rk:algass.cofibration}
In the context of Theorem \ref{algass.dk.model}, a cofibration $f : (\fC,A) \to (\fD,B) \in \algass$ in the Dwyer-Kan model category structure is exactly a map that satisfies the following two conditions.
\begin{enumerate}
\item $f$ is entrywise a cofibration in $\M$.
\item The categorical restriction $\fcat : \Acat \to \Bcat \in \Catm$ (Example \ref{ex:enriched.cat.extra}) is a cofibration in the Dwyer-Kan model category structure on $\Catm$ (Theorem \ref{dk.catm}).  By Remark \ref{rk:catm.cofibration} this means that:
\begin{enumerate}[(i)]
\item $f : \fC \to \fD$ is injective.
\item The induced functor $f_!\Acat \to \Bcat \in \Catmd$ is a cofibration in the projective model category structure on $\Catmd$ (Prop.\ref{catmc-model}).
\end{enumerate}
\end{enumerate}
This follows from the description of trivial fibrations in $\algass$ and in $\Catm$, namely, local trivial fibrations that are surjective on colors in both cases.
\end{remark}

\section{Homotopy Theory of Algebras over Operadic Collections}
\label{sec:alg.opcoll}

In this section, building upon the Dwyer-Kan model category structure in Theorem \ref{algass.dk.model}, we prove our first main result Theorem \ref{algso.dkmodel}, which equips the category of algebras over a suitable operadic collection with the Dwyer-Kan model category structure.  Throughout this section, suppose $\M$ is a bicomplete symmetric monoidal closed category, and $S \subseteq \bbNsq$ with $(1,1) \in S$.

\subsection{Augmented Operadic Collections}
Recall the concept of a map of operadic collections (Def. \ref{def:maps.operad.coll}) and the operadic collection $(S,\Ass)$ whose algebras are $\M$-enriched categories with $S$-entries (Example \ref{ex:enriched.cat.extra}).

\begin{definition}\label{def:aug.opcoll}
An \emph{augmented operadic collection}  in $\M$ is a triple $(S,\sO,\alpha)$ consisting of 
\begin{itemize}
\item an operadic collection $(S,\sO)$ in $\M$ and 
\item a map $\alpha : (S,\Ass) \to (S,\sO)$ of operadic collections 
\end{itemize}
such that, for each set $\fC$, the local restriction functor \eqref{restriction.c}
\[\nicexy{\alg(\Assc) & \alg(\osubc) \ar[l]_-{\alphac^{\ast}}}\]
creates and preserves filtered colimits.  In this case, we call $\alpha$ the  \emph{augmentation}.
\end{definition}

\begin{example}\label{ex:operad-like-augmented}
In the diagram in Example \ref{ex:maps.op.coll}, all the operadic collections, except for (unital linear) and $(S,\sI)$, are augmented.  Indeed, in each case, filtered colimits in $\alg(\osubc)$ are created and preserved by the forgetful functor to $\mtosofc$ by Prop. \ref{algebra-bicomplete}.
\end{example}

\begin{lemma}\label{algso.fil.colim}
Suppose $(S,\sO,\alpha)$ is an augmented operadic collection in $\M$.  Then the restriction functor \eqref{restriction.functor}
\[\nicexy{\algsass & \algso \ar[l]_-{\alpha^\ast}}\]
creates and preserves filtered colimits.
\end{lemma}

\begin{proof}
Any given filtered colimit in $\algso$ is computed in $\alg(\osubc)$ for some set $\fC$, and similarly for filtered colimits in $\algsass$.  The assertion now follows because the functor  $\alphac^* : \alg(\osubc) \to \algassc$ creates and preserves filtered colimits.
\end{proof}

\subsection{Fibrant Replacement}

Here we construct a fibrant replacement functor in the category of algebras over an augmented operadic collection.  This is one ingredient that we will need to establish the desired Dwyer-Kan model category structure on the algebra category.  First we define weak equivalences and fibrations in the algebra category.

\begin{definition}\label{def:algso.weq.fib}
Suppose $(S,\sO,\alpha)$ is an augmented operadic collection in a symmetric monoidal closed category $\M$ with a given model category structure.  
\begin{enumerate}
\item  A map $f \in \algso$ is a \emph{weak equivalence} (resp., \emph{fibration}) if $\alpha^*f \in \algsass$ is a weak equivalence (resp., fibration) in the sense of Def. \ref{def:algass.model}.
\item An object in $\algso$ is \emph{fibrant} if the unique map to the terminal object is a fibration.
\end{enumerate}
\end{definition}

\begin{remark}
In the context of Def. \ref{def:algso.weq.fib}, the restriction functor \eqref{restriction.functor}
\[\nicexy{\algsass & \algso \ar[l]_-{\alphastar}}\]
does not change the underlying objects.  In view of Remark \ref{rk:algass.weakeq}, the following statements are true.
\begin{enumerate}
\item $f \in \algso$ is a weak equivalence if and only if (i) $f$ is entrywise a weak equivalence in $\M$ and (ii) $\pizero (\alphastar f )^{\cat} \in \Cat$ is essentially surjective.
\item $f \in \algso$ is a fibration if and only if (i) $f$ is entrywise a fibration in $\M$ and (ii) $\pizero (\alphastar f )^{\cat} \in \Cat$ is an isofibration.
\end{enumerate}
\end{remark}

\begin{remark}\label{rk.algso.fibrant}
If $\M$ is as in Theorem \ref{dk.catm}, then an object in $\algso$ is fibrant if and only if it is locally fibrant because this property is true in $\algsass$ by Theorem \ref{algass.dk.model}.  Similarly, a map in $\algso$ is a trivial fibration if and only if it is a local trivial fibration that is surjective on colors.
\end{remark}

Recall the concept of a (lax symmetric monoidal) fibrant replacement functor from Def. \ref{def:fibrant.replacement}.

\begin{lemma}\label{algso.fib.replace}
Suppose $\M$ is as in Theorem \ref{dk.catm}, and $(S,\sO, \alpha)$ is an augmented operadic collection in $\M$.  If $\M$ has a lax symmetric monoidal fibrant replacement functor, then the category $\algso$ admits a fibrant replacement functor.
\end{lemma}

\begin{proof}
Suppose $(\fC,A) \in \algso$, so $A$ is an $\osubc$-algebra.  Write $\eta : \Id \to R$ for the given lax symmetric monoidal fibrant replacement functor in $\M$.  Applying $R$ entrywise to $A$ yields $RA \in \mtosofc$.  Since $R$ is lax symmetric monoidal, $RA$ naturally becomes an $R\osubc$-algebra, where $R\osubc$ is the $\sofc$-colored operad obtained by applying $R$ entrywise to $\osubc$ \cite{jy2} (Theorem 12.11).  Since $\eta$ is a monoidal natural transformation, applying it entrywise to $\osubc$ yields a map
\[\nicexy{\osubc \ar[r]^-{\etac} & R\osubc}\]
of $\sofc$-colored operads.  The induced functor on algebra categories
\[\nicexy{\algosubc & \algrosubc \ar[l]_-{\etacstar}}\]
does not change the underlying objects \cite{jy2} (Lemma 13.26).  The application of $\eta$ entrywise to $A$ now yields a map
\[\nicexy{A \ar[r]^-{\eta_A} & \etacstar(RA) \in \algosubc}\]
that is entrywise a weak equivalence in $\M$.

Using the map $\eta_A$, we will check that
\[\bigl(\fC,\etacstar(RA)\bigr) \in \algso\]
is the desired functorial fibrant replacement of $(\fC,A)$.  Functoriality is a direct consequence of the construction.  Let $\alpha$ be the augmentation of $(S,\sO)$.  Regarded in $\algso$, the map $\eta_A$ is the identity map on the color set $\fC$.  Since the restriction functor $\alpha^*$ \eqref{restriction.functor} does not change the underlying objects, the map $\alphastar\eta_A \in \algsass$ is a weak equivalence by Remark \ref{rk:algass.weakeq}.  So the map $\eta_A \in \algso$ is a weak equivalence.

It remains to show that $\etacstar(RA) \in \algso$ is a fibrant object, i.e., locally fibrant.  Since $\etacstar$ does not change the underlying objects, we just need to observe that $RA$ is locally fibrant 
\end{proof}

\subsection{Isofibration Axiom}
The following condition will be needed later in the main theorem.

\begin{definition}\label{def:cat.fib.axiom}
Suppose $\M$ is a symmetric monoidal closed category with a given model category structure.
\begin{enumerate}
\item A \emph{simple fibration} in $\Catm$ is a map $f : (\fC,A) \to (\fC,B) \in \Catm$ such that:
\begin{enumerate}
\item On objects $f$ is the identity map.
\item For each pair of objects $x$ and $y$ in $A$, the map $f : A(x,y) \to B(x,y)$ is a fibration between fibrant objects in $\M$.
\end{enumerate}
\item We say that $\M$ satisfies the \emph{isofibration axiom} if each simple fibration in $\Catm$ is a homotopical isofibration in the sense of Def. \ref{def:pizero}.
\end{enumerate}
\end{definition}

\begin{example}
The category of simplicial sets with the usual Kan-Quillen model category structure \cite{quillen.homotopical} (II.3 Theorem 3) satisfies the isofibration axiom.  This is a consequence of basic properties of Kan fibrations.
\end{example}

\begin{example}
On the other hand, in general the isofibration axiom does not come for free.  For example,  
the monoidal category of sets with the trivial model category structure (in which weak equivalences are the bijections and every map is both a fibration and a cofibration) does \underline{not} satisfy the isofibration axiom.  Indeed, in this case a simple fibration is a functor that is the identity map on objects.  The functor $\pizero : \Cat \to \Cat$ is the identity functor.  So a simple fibration $f : A \to B \in \Cat$ is a homotopical isofibration if and only if, for each pair of objects $x$ and $y$ in $A$, every isomorphism in $B(x,y)$ can be lifted back to $A(x,y)$ via $f$.  This is not automatic. 

For instance, take $A$ to be the category with two objects $\{0,1\}$ and no non-identity maps.  Take $B$ to be the category with objects $\{0,1\}$ whose only non-identity maps are an isomorphism $v \in B(0,1)$ and its inverse $v^{-1} \in B(1,0)$.  Then the object-preserving functor $f : A \to B$ is a simple fibration but not a homotopical isofibration because the isomorphism $v \in B(0,1)$ cannot be lifted back to $A(0,1) = \varnothing$ via $f$.
\end{example}

In Prop. \ref{isofibration.axiom.extend} below we will develop a simple criterion for checking the isofibration axiom, for which we first need some definitions.  An adjunction with left adjoint $L$ and right adjoint $R$ is denoted by $(L,R)$.  

\begin{definition}\label{def:quillen.adjunction}
Suppose $L : \M \adjoint \N : R$ is an adjunction between model categories.
\begin{enumerate}
\item Recall that $(L,R)$ is called a \emph{Quillen adjunction} if the right adjoint preserves both fibrations and trivial fibrations.  In this case, $L$ (resp., $R$) is called a \emph{left} (resp., \emph{right}) \emph{Quillen functor}.
\item Suppose that $(L,R)$ is a Quillen adjunction between monoidal categories.  Then $(L,R)$ is called a \emph{mild monoidal Quillen adjunction} if the following conditions hold:
\begin{enumerate}
\item The right adjoint $R$ is lax monoidal (Def. \ref{def:monoidal.functor}).
\item For some cofibrant replacement $q : Q\tensorunitm \to \tensorunitm$ of the monoidal unit in $\M$, the composite
\[\nicexy{LQ\tensorunitm \ar[r]^-{Lq} & L\tensorunitm \ar[r]^-{\Rbar^0} & \tensorunitn}\]
is a weak equivalence in $\N$, in which $\Rbar^0$ is the adjoint of the structure map $R^0 : \tensorunitm \to R\tensorunitn$.
\end{enumerate}
\end{enumerate}
\end{definition}

The following observation is useful for checking the isofibration axiom.

\begin{proposition}\label{isofibration.axiom.extend}
Suppose $L : \M \adjoint \N : R$ is a mild monoidal Quillen adjunction between monoidal categories.  If $\M$ satisfies the isofibration axiom, then so does $\N$.
\end{proposition}

\begin{proof}
Suppose $f : (\fC,A) \to (\fC,B) \in \Catn$ is a simple fibration, so it is the identity map on colors and is entrywise a fibration between fibrant objects in $\N$.  We must show that $\pizero(f) \in \Cat$ is an isofibration.  Since the right Quillen functor $R$ is lax monoidal, applying it entrywise to $f$ yields a simple fibration
\[\nicexy{(\fC, RA) \ar[r]^-{Rf} & (\fC,RB) \in \Catm.}\]
So $\pizero(Rf) \in \Cat$ is an isofibration.  For any pair of colors $c,d \in \fC$, write $\dcsingle$ for $(c,d)$ and consider the commutative diagram
\[\begin{small}\nicexy@C-.8cm{\pizero(A)\dcsingle = \Ho(\N)\bigl(\tensorunitn, A\dcsingle\bigr) \ar[d]_-{\pizero(f)} \ar[r]^-{\cong} & \Ho(\N)\bigl(LQ\tensorunitm, A\dcsingle\bigr) \cong \Ho(\M)\bigl(\tensorunitm, RA\dcsingle \bigr) = \pizero(RA)\dcsingle  \ar[d]^-{\pizero(Rf)}\\
\pizero(B)\dcsingle  = \Ho(\N)\bigl(\tensorunitn, B\dcsingle \bigr) \ar[r]^-{\cong} & \Ho(\N)\bigl(LQ\tensorunitm, B\dcsingle \bigr) \cong \Ho(\M)\bigl(\tensorunitm, RB\dcsingle \bigr) = \pizero(RB)\dcsingle}\end{small}\]
of sets.  The top isomorphisms are induced by the weak equivalence $LQ\tensorunitm \to \tensorunitn$ and the derived adjuction of the Quillen adjunction $(L,R)$, respectively, and similarly for the bottom isomorphisms.  Since $\pizero(Rf)$ is an isofibration, so is $\pizero(f)$.
\end{proof}

\begin{example}\label{ex:isofibration.adjunction}
Recall that $\sset$, the category of simplicial sets with the Kan-Quillen model category structure, satisfies the isofibration axiom.
\begin{enumerate}
\item 
Consider the mild monoidal Quillen adjunction \cite{rezk} (6.1)
\[\nicexy{\sset \ar@<2pt>[r]^-{\pi} & \Cat \ar@<2pt>[l]^-{\mu}}\]
between simplicial sets and $\Cat$ with the folk model structure.  Here $\pi$ sends a simplicial set to its fundamental groupoid, and $\mu$ sends a category to the simplicial nerve of the subcategory of isomorphisms.  Then Prop. \ref{isofibration.axiom.extend} applies to show that $\Cat$ satisfies the isofibration axiom.
\item
Suppose $\smod(k)$ is the monoidal model category of simplicial $k$-modules for some commutative unital ring $k$ \cite{quillen.homotopical} (II.4 Theorem 4).  The mild monoidal Quillen adjunction
\[\nicexy{\sset \ar@<2pt>[r] & \smod(k) \ar@<2pt>[l]}\]
is induced levelwise by the free-forgetful adjunction between sets and $k$-modules.  Then Prop. \ref{isofibration.axiom.extend} applies to show that $\smod(k)$ satisfies the isofibration axiom.
\item 
With $k$ as above, suppose $\chk$ is the monoidal model category of non-negative graded chain complexes of $k$-modules; see \cite{ds} or \cite{quillen.homotopical} (4.11-4.12).  Suppose  
\begin{equation}\label{dold-kan}
\nicexy{\smod(k) \ar@<2pt>[r]^-{N} & \chk \ar@<2pt>[l]^-{\Gamma}}
\end{equation}
is the Dold-Kan correspondence \cite{dold,kan}.  Then Prop. \ref{isofibration.axiom.extend} and the previous case show that $\chk$ satisfies the isofibration axiom.
\item
With $k$ as above, suppose $\Chk$ is the monoidal model category of unbounded chain complexes of $k$-modules \cite{hovey} (2.3).  Suppose  
\[\nicexy{\chk \ar@<2pt>[r]^-{i} & \Chk \ar@<2pt>[l]^-{C_0}}\]
is the mild monoidal Quillen adjunction such that $i$ is the inclusion functor and $C_0$ is the connective cover functor \cite{shipley07}.  Then Prop. \ref{isofibration.axiom.extend} and the previous case show that $\Chk$ satisfies the isofibration axiom.
\end{enumerate}
\end{example}

\subsection{Path Objects}
Next we construct path objects in the category of algebras over an augmented operadic collection.  This is another ingredient for the main theorem.

\begin{definition}\label{def:path.object}
Suppose $\calc$ is a category equipped with a class of maps called \emph{weak equivalences} and another class of maps called \emph{fibrations}.  A \emph{path object} for an object $X$ in $\calc$ is a factorization of the diagonal map
\[\nicexy{X \ar[r]^-{\sim} \ar@{-<} `u[rr] `[rr]|-{\text{diagonal}} [rr]
& P(X) \ar@{>>}[r] & X \times X}\]
in $\calc$ into a weak equivalence followed by a fibration \cite{hirschhorn} (7.3.2(3)).
\end{definition}

For an augmented operadic collection $(S,\sO,\alpha)$ (Def. \ref{def:aug.opcoll}), recall the definitions of weak equivalences and fibrations in $\algso$ in Def. \ref{def:algso.weq.fib}.

\begin{lemma}\label{algso.path}
Suppose $\M$ is as in Theorem \ref{dk.catm} that also satisfies the isofibration axiom (Def. \ref{def:cat.fib.axiom}) and has a cocommutative comonoidal interval \eqref{comonoidal.interval}.  Suppose $(S,\sO,\alpha)$ is an augmented operadic collection in $\M$.  Then every fibrant object in $\algso$ has a path object that is preserved by the restriction functor \eqref{restriction.functor}
\[\nicexy{\algsass & \algso \ar[l]_-{\alpha^\ast}.}\]
\end{lemma}

\begin{proof}
Suppose $(\fC,A) \in \algso$ is a fibrant object, so $A$ is an $\osubc$-algebra that is entrywise a fibrant object in $\M$ (Remark \ref{rk.algso.fibrant}).  Taking the entrywise internal hom from the given cocommutative comonoidal interval into $A$ yields a factorization
\begin{equation}\label{ca.path}
\nicexy{A \cong A^{\tensorunit} \ar[r]^-{\betastar} \ar@{-<} `u[rr] `[rr]|-{\text{diagonal}} [rr]
& A^H \ar[r]^-{\epsilonstar} & A^{\tensorunit \amalg \tensorunit} \cong A \times A}
\end{equation}
of the diagonal map of $A \in \mtosofc$.  Here we use the notation $X^Y$ for the internal hom object $\Homm(Y,X)$ in $\M$.  We will show that the factorization \eqref{ca.path} can be upgraded to a path object for $(\fC,A)$.  Specifically, we will equip each of the objects $A^H$ and $A^{\oneone}$ with an $\osubc$-algebra structure in such a way that $\betastar$ is a weak equivalence and $\epsilonstar$ is a fibration in $\algso$.

Since the assumed comonoidal interval is cocommutative,  taking the entrywise internal hom from it into the $\sofc$-colored operad $\osubc$ yields maps
\begin{equation}\label{interval.to.osubc}
\nicexy{\osubc \cong \osubc^{\tensorunit} \ar[r]^-{\betastaro} & \osubc^H \ar[r]^-{\epsilonstaro} & \osubc^{\tensorunit \amalg \tensorunit}}
\end{equation}
of $\sofc$-colored operads.  By the naturality of internal hom and the fact that $H$ is a counital coassociative cocommutative comonoid, $A^H$ naturally becomes an $\osubc^H$-algebra.  By restricting along the map $\betastaro$ in \eqref{interval.to.osubc}, $A^H$ becomes an $\osubc$-algebra such that the map $\betastar : A \to A^H$ in \eqref{ca.path} is now a map of $\osubc$-algebras.

To show that the map
\[\nicexy{(\fC,A) \ar[r]^-{\betastar} & (\fC,A^H) \in \algso}\] 
is a weak equivalence, we need to show that $\alphastar \betastar \in \algsass$ is a weak equivalence in the sense of Def. \ref{def:algass.model}.  Since the restriction functor $\alphastar$ does not change the underlying map of $\betastar$ and the latter is the identity map on colors, it is enough to show that $\betastar$ is entrywise a weak equivalence in $\M$.  To see that $\betastar$ is entrywise a weak equivalence, it is enough to observe that $\beta : H \to \tensorunit$ is a weak equivalence between cofibrant objects and that $A$ is entrywise fibrant and to use the pushout product axiom (Def. \ref{def:ppaxiom}) and Ken Brown's Lemma \cite{hovey} (1.1.12).

Similarly, $A \times A \cong A^{\tensorunit \amalg \tensorunit}$ naturally becomes an $\osubc^{\tensorunit \amalg \tensorunit}$-algebra.  By restricting along the map $\epsilonstaro \betastaro$ in \eqref{interval.to.osubc}, $A \times A$ becomes an $\osubc$-algebra such that the map $\epsilonstar$ in \eqref{ca.path} is now a map of $\osubc$-algebras.  It remains to show that
\[\nicexy{(\fC,A^H) \ar[r]^-{\epsilonstar} & (\fC, A \times A) \in \algso}\]
is a fibration, i.e., that $\alphastar \epsilonstar \in \algsass$ is a fibration in the sense of Def. \ref{def:algass.model}.  Since $\epsilon : \oneone \to H$ is a cofibration between cofibrant objects and $A$ is entrywise fibrant, the map $\epsilonstar$ in \eqref{ca.path} is entrywise a fibration between fibrant objects in $\M$ by the pushout product axiom (Def. \ref{def:ppaxiom}).  Since the restriction functor $\alphastar$ does not change the underlying map, $\alphastar \epsilonstar$ is entrywise a fibration.  Furthermore, the restriction to the categorical part $(\alphastar \epsilonstar)^{\cat} \in \Catm$ is a simple fibration (Def. \ref{def:cat.fib.axiom}).  The isofibration axiom in $\M$ now implies that $(\alphastar \epsilonstar)^{\cat}$ is a homotopical isofibration.  Therefore, the map $\alphastar \epsilonstar \in \algsass$ is a fibration.

With the $\osubc$-algebra structure above, we have shown that the factorization in \eqref{ca.path} is a path object for $(\fC,A) \in \algso$.  Since weak equivalences and fibrations in $\algso$ are defined by their $\alphastar$-restrictions (Def. \ref{def:algso.weq.fib}), this path object is preserved by the restriction functor $\alphastar$.
\end{proof}

\subsection{Dwyer-Kan Model Structure}
We will use the following model category structure lifting result from \cite{jy} (3.3).  

\begin{lemma}\label{model.lift}
Suppose $L : \M \adjoint \N : R$ is an adjunction such that:
\begin{itemize}
\item $\M$ is a strongly cofibrantly generated model category with set of generating (trivial) cofibrations $I$ (resp., $J$).
\item $\N$ has all small limits and colimits, and $R$ creates and preserves filtered colimits.
\item A map $f \in \N$ is called a weak equivalence (resp., fibration) if and only if $Rf \in \M$ is a  weak equivalence (resp., fibration).  An object in $\N$ is called fibrant if and only if the unique map to the terminal object is a fibration.
\item There is a fibrant replacement functor in $\N$ (Def. \ref{def:fibrant.replacement}).
\item Every fibrant object in $\N$ has a path object (Def. \ref{def:path.object}) that is preserved by $R$.
\end{itemize}
Then $\N$ admits a cofibrantly generated model category structure with the above weak equivalences and fibrations such that:
\begin{itemize}
\item The set of generating (trivial) cofibrations in $\N$ is $LI$ (resp., $LJ$).
\item $(L,R)$ is a Quillen adjunction.
\end{itemize}
\end{lemma}

\begin{definition}\label{def:convenient}
A \emph{convenient model category} is a combinatorial monoidal model category $\M$ that satisfies the following conditions:
\begin{enumerate}
\item The class of weak equivalences is closed under filtered colimits.
\item $\M$ satisfies the monoid axiom and the isofibration axiom (Def. \ref{def:cat.fib.axiom}), has a cofibrant monoidal unit, and is right proper.
\item $\M$ has a lax symmetric monoidal fibrant replacement functor (Def. \ref{def:fibrant.replacement}) and a cocommutative comonoidal interval (Def. \ref{def:comon.interval}).
\end{enumerate}
\end{definition}

\begin{example}\label{ex:convenient.examples}
Examples of convenient model categories include the categories of simplicial sets  \cite{quillen.homotopical}, of (non-negatively graded or unbounded) chain complexes of modules over a characteristic zero field $k$ \cite{hovey}, of simplicial $k$-modules \cite{quillen.homotopical}, and of small categories with the folk model structure \cite{jt91,rezk}.
\end{example}

\begin{definition}\label{def:opcoll.admissible}
Suppose $\M$ is a symmetric monoidal closed category with a given model category structure, and $(S,\sO,\alpha)$ is an augmented operadic collection in $\M$ (Def. \ref{def:aug.opcoll}).  We will say that $(S,\sO,\alpha)$ is \emph{admissible} if the category $\algso$ admits the model category structure with weak equivalences and fibrations as in Def. \ref{def:algso.weq.fib}.  This is called the \emph{Dwyer-Kan model category structure} on $\algso$.
\end{definition}

The following existence result is our first main result.

\begin{theorem}\label{algso.dkmodel}
Suppose $\M$ is a convenient model category and $(S,\sO, \alpha)$ is an augmented operadic collection in $\M$.  Then $(S,\sO,\alpha)$ is admissible.  Furthermore:
\begin{itemize}
\item The adjunction $\alpha_! : \algsass \adjoint \algso : \alpha^*$ in Prop. \ref{opcoll.map.leftadj} is a Quillen adjunction.
\item $\algso$ is cofibrantly generated with set of generating (trivial) cofibrations $\alpha_!(I^S)$ (resp., $\alpha_!(J^S)$), where $I^S$ and $J^S$ are as in \eqref{IsJs}.
\item A map in $\algso$ is a trivial fibration if and only if it is a local trivial fibration (Def. \ref{def:algass.model}) that is also surjective on colors.  
\item An object in $\algso$ is fibrant if and only if it is locally fibrant.
\end{itemize}
\end{theorem}

\begin{proof}
Since $\M$ is a convenient model category, by Theorem \ref{algass.dk.model} the category $\algsass$ of $\M$-enriched categories with $S$-entries admits the Dwyer-Kan model category structure.  Furthermore, it is strongly cofibrantly generated with set of  generating (trivial) cofibrations $I^S$ (resp., $J^S$).  Define weak equivalences and fibrations in $\algso$ as in Def. \ref{def:algso.weq.fib}.

We will apply the model category lifting Lemma \ref{model.lift} to the adjunction 
\[\alpha_! : \algsass \adjoint \algso : \alpha^*\]
in Prop. \ref{opcoll.map.leftadj}.  The algebra category $\algso$ has all small limits and colimits by Prop. \ref{algso.bicomplete}.  The right adjoint $\alphastar$ creates and preserves filtered colimits by Lemma \ref{algso.fil.colim}.  There is a fibrant replacement functor in $\algso$ by Lemma \ref{algso.fib.replace}.  Finally, every fibrant object in $\algso$ has a path object that is preserved by $\alphastar$ by Lemma \ref{algso.path}.  Therefore, Lemma \ref{model.lift} applies to the adjunction $(\alpha_!,\alphastar)$.  The descriptions of trivial fibrations and fibrant objects are true because they are true in $\algsass$.
\end{proof}

\begin{example}[Dwyer-Kan Model Categories of Operad-Like Structure]
\label{ex:gprop.dwyer.kan}
Suppose $\M$ is a convenient model category, such as one of those in Example \ref{ex:convenient.examples}, and $(S,\sO, \alpha)$ is any one of the augmented operadic collections in $\M$ in the diagram in Example \ref{ex:maps.op.coll}, except for (unital linear) and $(S,\sI)$.  Then Theorem \ref{algso.dkmodel} implies that the category $\algso$, which is one of the categories in the diagram in Example \ref{ex:algebras.op.coll} except for (enriched categories) and $\mtosdash$, admits the Dwyer-Kan model category structure.  For example, taking $(S,\sO,\alpha)$ to be the augmented operadic collection of all wheeled graphs (resp., connected wheeled graphs, or wheeled trees), we have that the category of all small wheeled props (resp., wheeled properads, or wheeled operads) in $\M$ admits the Dwyer-Kan model category structure.

Using different methods, when $\M$ is the category of simplicial sets and $(S,\sO)$ is the operadic collection of unital trees (resp., wheel-free graphs, connected wheel-free graphs, or simply-connected graphs), whose algebras are colored operads (resp., props, properads, or dioperads), this Dwyer-Kan model category structure was first obtained in \cite{cm13,robertson} (resp., \cite{hr16} and \cite{hry.properad}).  For more general monoidal model categories $\M$, this Dwyer-Kan model category structure was obtained in \cite{cav14} for colored operads and in \cite{cav15} for props.  Our Dwyer-Kan model category structure in the wheeled cases (i.e., wheeled props, wheeled properads, and wheeled operads) are new.
\end{example}

\begin{remark}\label{rk:integral.model}
When $\M$ is a convenient model category, every colored operad in $\M$ is admissible (Def. \ref{operad.alg.model}); see Example \ref{ex:bm-admissible}.  Recall that the category $\algso$ coincides with the Grothendieck construction $\int \Obar$ \eqref{op.coll.obar}.  Equip the category of sets with the trivial model category structure, in which a weak equivalence is a bijection and every map is both a fibration and a cofibration.  Then \cite{hp} (3.0.12) says that the category $\int \Obar = \algso$ admits a model category structure, called the \emph{integral model structure}, in which a map $f$ is a weak equivalence if and only if it is a bijection on color sets and it is entrywise a weak equivalence.  So there are fewer weak equivalences in the integral model structure than in our Dwyer-Kan model category structure on $\algso$.  A fibration in the integral model structure is an entrywise fibration, so there are more fibrations in the integral model structure than there are in the Dwyer-Kan model category structure.  In fact, using any one of the nine model category structures on the category of sets \cite{camarena}, the associated integral model structure is still different from our Dwyer-Kan model category structure.
\end{remark}

\subsection{Quillen Adjunctions}
Recall from Def. \ref{def:quillen.adjunction} the concept of a Quillen adjunction.

\begin{corollary}\label{dk.quillen}
Suppose $\M$ is a convenient model category, $(S^1,\sO^1, \alpha^1)$ and $(S^2,\sO^2, \alpha^2)$ are augmented operadic collections in $\M$, and $\alpha : (S^1,\sO^1) \to (S^2,\sO^2)$ is a map of operadic collections such that the diagram of operadic collections
\[\nicexy{(S^1,\sO^1) \ar[r]^-{\alpha} & (S^2,\sO^2)\\
(S^1,\Assone) \ar[u]^-{\alpha^1} \ar[r]^-{\iota} & (S^2,\Asstwo) \ar[u]_-{\alpha^2}}\]
is commutative, where $\iota$ is the map in \eqref{ass.onetwo}.  Then the adjunction 
\[\nicexy{\algsoone \ar@<2pt>[r]^-{\alpha_!} & \algsotwo \ar@<2pt>[l]^-{\alpha^\ast}}\]
in Prop. \ref{opcoll.map.leftadj} is a Quillen adjunction, where each category $\alg(S^i,\sO^i)$ is equipped with the Dwyer-Kan model category structure in Theorem \ref{algso.dkmodel}.
\end{corollary}

\begin{proof}
The diagram
\[\nicexy{\algsoone \ar[d]_-{\alpha^{1\ast}} & \algsotwo \ar[l]_-{\alphastar} \ar[d]^-{\alpha^{2\ast}}\\
\algsassone & \algsasstwo \ar[l]_-{\iota^\ast}}\]
of restriction functors is commutative.  By the definition of the Dwyer-Kan model category structure (Def. \ref{def:algso.weq.fib}), the right adjoint $\alphastar$ preserves both fibrations and weak equivalences, hence also trivial fibrations.
\end{proof}

\begin{example}
Suppose $\M$ is a convenient model category.  Corollary \ref{dk.quillen} implies that in the following diagram extracted from Example \ref{ex:algebras.op.coll}, 
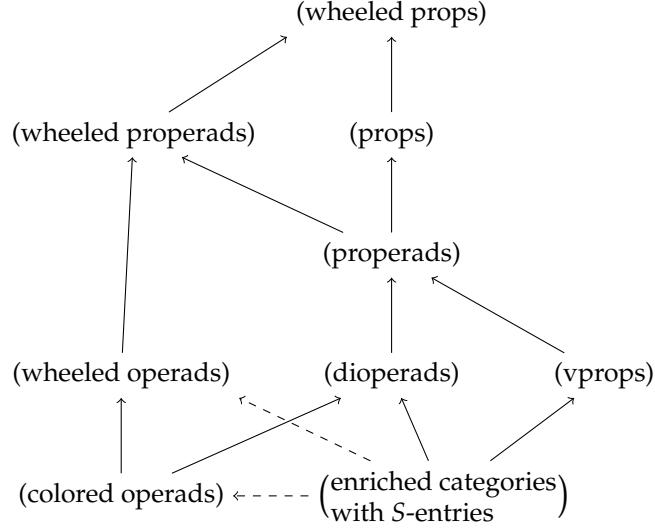
\begin{center}
\begin{small}
\begin{tikzpicture}
\node (whgr) {(wheeled props)};
\node[below=of whgr] (whfree) {(props)};
\node[left=of whfree] (cwhgr) {(wheeled properads)};
\node[below=of whfree] (cwhfree) {(properads)};
\node[below=of cwhfree] (simply) {(dioperads)};
\node[left=of simply] (whtree) {(wheeled operads)};
\node[below right=of cwhfree] (graftcor) {(vprops)};
\node[below=of whtree] (utree) {(colored operads)};
\node[right=of utree] (ass) {$\Bigl($\parbox{3cm}{enriched categories with $S$-entries}$\Bigr)$};
\draw[->] (cwhgr)--(whgr.south west); \draw[->] (whfree)--(whgr); 
\draw[->] (cwhfree)--(cwhgr);
\draw[->] (cwhfree)--(whfree); \draw[->] (whtree)--(cwhgr);
\draw[->] (graftcor)--(cwhfree); \draw[->] (ass)--(graftcor);
\draw[->] (simply)--(cwhfree); \draw[->] (utree)--(simply); 
\draw[->] (utree)--(whtree);
\draw[->] (ass)--(simply); \draw[dashed,->] (ass)--(utree); 
\draw[dashed,->] (ass)--(whtree.south east);
\end{tikzpicture}
\end{small}
\end{center}
each arrow is a left Quillen functor between Dwyer-Kan model categories of algebras over some operadic collections in $\M$.  The dashed maps were explained in Example \ref{ex:maps.op.coll}.
\end{example}

\section{Lifting Quillen Equivalences to Algebra Categories}
\label{sec:lifting.quillen}

In this section we show that in the context of Theorem \ref{algso.dkmodel}, the Dwyer-Kan model category structure on the algebra category $\algso$ is  homotopy invariant with respect to a simultaneous change of the underlying category and the operadic collection.  The underlying category changes via a suitable Quillen equivalence, and the operadic collection changes via a map of operadic collections that is entrywise a weak equivalence.  To achieve this kind of results, some extra cofibrancy conditions are needed, which can be imposed upon either the Quillen equivalence or the operadic collections.  Therefore, as we will explain below, there are two versions of the main result in this section.

\subsection{Lifting Quillen Adjunctions}
We first consider lifting a Quillen adjunction between the underlying categories to the categories of algebras over operadic collections.

\begin{assumption}\label{lifting.qadj.setting}
Suppose: 
\begin{enumerate}
\item $L : \M \adjoint \N : R$ is an adjunction between bicomplete symmetric monoidal closed categories with $R$ lax symmetric monoidal.
\item $(S,\sO,\alphaone)$ is an augmented operadic collection in $\M$ (Def. \ref{def:aug.opcoll}), and $(S,\sP,\alphatwo)$ is an augmented operadic collection in $\N$ with the same $S$.
\item $\alpha : (S,\sO) \to (S,R\sP)$ is a map of operadic collections in $\M$ that is compatible with the augmentations in the sense that the diagram of operadic collections in $\M$
\begin{equation}\label{augmentation.compatible}
\nicexy{(S,\sO) \ar[r]^-{\alpha} & (S,R\sP)\\
(S,\Ass) \ar[u]^-{\alphaone} \ar[r]^-{\epsilon} & (S, R\Ass) \ar[u]_-{R\alphatwo}}
\end{equation}
is commutative.  Here, for each set $\fC$,
\[(R\sP)_{\fC} = R(\sP_{\fC})\]
is the $\sofc$-colored operad obtained by applying $R$ entrywise to $\psubc$, as in the proof of Lemma \ref{algso.fib.replace}, and similarly for $R\Ass$. The map $\epsilonc$ is entrywise either the unique map $\varnothing \to R\varnothing$ or is isomorphic to the adjoint of a finite coproduct of the structure map $L\tensorunitm \to \tensorunitn$.
\end{enumerate}
\end{assumption}

\begin{definition}\label{lbar.exists}
Under Assumption \ref{lifting.qadj.setting}, consider the solid-arrow diagram of functors
\begin{equation}\label{lbar.algso.diagram}
\nicexy@C+1cm{\algso \ar@{.>}@<2.5pt>[r]^-{\Lbar} \ar@<2.5pt>[d]^-{U}
& \algsp \ar@<2.5pt>[l]^-{\Ralpha} \ar@<2.5pt>[d]^-{U} \\
\mtosdash \ar@<2.5pt>[r]^-{L} \ar@<2.5pt>[u]^-{\sO_{(-)} \comp -}  
& \ntosdash \ar@<2.5pt>[l]^-{R} \ar@<2.5pt>[u]^-{\sP_{(-)} \comp -}}
\end{equation}
with $\mtosdash$ and $\ntosdash$ as in Example \ref{ex:initial.collection}.  The vertical free-forgetful adjunctions (Prop. \ref{opcoll.map.leftadj}) are induced by the maps \eqref{opcoll.from.initial}.  The original adjunction $(L,R)$ is prolonged entrywise to the bottom horizontal adjunction between $\mtosdash$ and $\ntosdash$.  Define the functor $\Ralpha$ by sending $(\fC,B) \in \algsp$ (so $B$ is a $\psubc$-algebra) to
\begin{equation}\label{Ralpha}
\Ralpha(\fC,B) = (\fC,RB) \in \algso
\end{equation}
in which:
\begin{enumerate}
\item The underlying object of $RB$ is the result of applying $R$ entrywise to $B$.  
\item The $\osubc$-algebra structure maps on $RB$ are obtained by pulling back the structure maps of $RB \in \alg(R\psubc)$ along the map $\alphac : \osubc \to R\psubc$ of $\sofc$-colored operads. 
\end{enumerate}
\end{definition}

\begin{lemma}\label{ralpha.left.adj}
Under Assumption \ref{lifting.qadj.setting} the functor $\Ralpha$ in \eqref{lbar.algso.diagram} admits a left adjoint $\Lbar$ such that there is a natural isomorphism
\[\nicexy{\Lbar \bigl(\sO_{(-)} \comp -\bigr) \cong \bigl(\sP_{(-)} \comp -\bigr) L.}\]
\end{lemma}

\begin{proof}
In the diagram \eqref{lbar.algso.diagram}, there is an equality
\[U \Ralpha = RU.\]  Moreover, both vertical right adjoints $U$ are monadic in the sense of \cite{borceux} (4.4.1), and the category $\algsp$ is cocomplete by Prop. \ref{algso.bicomplete}.  So the left adjoint $\Lbar$ of $\Ralpha$ exists by the Adjoint Lifting Theorem \cite{borceux} (4.5.6).  The commutativity of the left adjoint diagram follows from that of the right adjoint diagram and the uniqueness of left adjoints.
\end{proof}

\begin{lemma}\label{lifting.qadj}
Under Assumption \ref{lifting.qadj.setting} suppose further that $L : \M \adjoint \N : R$ is a mild monoidal Quillen adjunction (Def. \ref{def:quillen.adjunction}) between convenient model categories (Def. \ref{def:convenient}).
Then the adjunction
\begin{equation}\label{lbar.ralpha.qadj}
\nicexy{\algso \ar@<2.5pt>[r]^-{\Lbar} & \algsp \ar@<2.5pt>[l]^-{\Ralpha}}
\end{equation}
in Lemma \ref{ralpha.left.adj} is a Quillen adjunction, in which both $\algso$ and $\algsp$ are equipped with the Dwyer-Kan model category structure in Theorem \ref{algso.dkmodel}.
\end{lemma}

\begin{proof}
According to \cite{dugger} (A.2) (or equivalently \cite{hirschhorn} 8.5.4), to show that $\Ralpha$ is a right Quillen functor, it is enough to show that it preserves fibrations between fibrant objects and all trivial fibrations.  In both $\algso$ and $\algsp$ a trivial fibration is a local trivial fibration that is surjective on colors.  The functor $\Ralpha$ does not change what a map does on colors.  On the underlying entries, $\Ralpha$ is simply $R$ entrywise, which preserves (trivial) fibrations because $R$ is a right Quillen functor.  So $\Ralpha$ preserves trivial fibrations, local fibrations, and fibrant objects.  It remains to show that $\Ralpha$ preserves fibrations between fibrant objects.

To this end, recall the functor $(-)^{\cat}$ that restricts to the categorical part \eqref{enriched.cat.adjoint}.  Suppose $(\fD,B) \in \algsp$ is fibrant, so $B$ is an entrywise fibrant $\psubd$-algebra.  To simplify the notation, we will write $\Bcat \in \Catn$ for $(\alphad^{2*} B)^{\cat}$ and similarly for $(S,\sO)$-algebras.  This abbreviation should not cause any confusion because the local restriction functor $\alphad^{2*}$ \eqref{restriction.c} does not change the underlying objects.  Since $R$ is lax monoidal, applying it entrywise to $\Bcat$ yields $R\Bcat \in \Catm$.  The augmentation compatibility diagram \eqref{augmentation.compatible} implies that there is an equality
\[(\Ralpha B)^{\cat} = R\Bcat \in \Catm.\]
Recall the functor $\pizero$ in \eqref{functor.pizero}.  Using the fact that $B$ is entrywise fibrant, for any pair of objects $x, y \in \fD$, we have bijections:
\begin{equation}\label{pizero.rb}
\begin{split}
\pizerom\bigl((\Ralpha B)^{\cat}\bigr)(x,y) 
&= \pizerom\bigl(R\Bcat\bigr)(x,y)\\
&= \Ho(\M)\bigl(\tensorunitm, R\Bcat(x,y)\bigr)\\
&\cong \Ho(\N)\bigl(LQ\tensorunitm, \Bcat(x,y)\bigr)\\
&\cong \Ho(\N)\bigl(\tensorunitn, \Bcat(x,y)\bigr)\\
&= \pizeron(\Bcat)(x,y).
\end{split}
\end{equation}
For the two isomorphisms, we used the fact that $(L,R)$ is a Quillen adjunction and the fact that $LQ\tensorunitm \to \tensorunitn$ is a weak equivalence in $\N$, respectively.

Suppose $f \in \algsp$ is a fibration between fibrant objects.  Since $f$ is, in particular, a local fibration between fibrant objects, we already observed above that $\Ralpha f$ is a local fibration between fibrant objects.  It remains to show that $(\Ralpha f)^{\cat}$ is a homotopical isofibration.  By assumption $\pizeron(\fcat) \in \Cat$ is an isofibration (Def. \ref{def:pizero}).  The calculation \eqref{pizero.rb} implies that $\pizerom\bigl((\Ralpha f)^{\cat}\bigr) \in \Cat$ is also an isofibration; i.e., $(\Ralpha f)^{\cat} \in \Catm$ is a homotopical isofibration.
\end{proof}

\begin{example}
Lemma \ref{lifting.qadj} applies to all the adjunctions in Example \ref{ex:isofibration.adjunction} as long as $k$ is a field of characteristic zero.
\end{example}

\subsection{Cofibrant Algebras}
In the proof of the main result of this section, we will need to know that a cofibrant algebra $(\fC,A)$ over an operadic collection $(S,\sO)$ is also cofibrant as an $\osubc$-algebra.  To make sense of this, we first need to define the corresponding model category structure on $\osubc$-algebras.

\begin{definition}\label{operad.alg.model}
Suppose $\M$ is a symmetric monoidal closed category with a given model category structure, and $\sO$ is a $\fC$-colored operad in $\M$ for some set $\fC$ (Def. \ref{def:colored-operad}).  Following \cite{bm03} we will say that $\sO$ is \emph{admissible} if the category $\algo$ of $\sO$-algebras (Def. \ref{colored-operad-algebra}) admits the model category structure in which a map is a weak equivalence (resp., fibration) if and only if its underlying map is entrywise a weak equivalence (resp., fibration) in $\M$.  This is called the \emph{projective model category structure} on $\algo$.
\end{definition}

\begin{example}\label{ex:bm-admissible}
If $\M$ is a convenient model category  (Def. \ref{def:convenient}), then every colored operad in $\M$ is admissible by \cite{bm07} (2.1).
\end{example}

\begin{lemma}\label{cofibrant.algebra}
Suppose $\M$ is a convenient model category, $(S,\sO,\alpha)$ is an augmented operadic collection in $\M$ (Def.  \ref{def:aug.opcoll}), and $(\fC,A) \in \algso$ is cofibrant in the Dwyer-Kan model category structure (Def. \ref{def:opcoll.admissible}).  Then $A \in \algosubc$ is cofibrant in the projective model category structure.
\end{lemma}

\begin{proof}
Recall that the augmented operadic collection $(S,\sO,\alpha)$ is admissible by Theorem \ref{algso.dkmodel} and that $\osubc$ is an $\sofc$-colored operad in $\M$ (Def. \ref{def:operadic.collection}).  Suppose given a solid-arrow diagram
\begin{equation}\label{algso.cof.lifting}
\nicexy{\varnothing \ar[d] \ar[r] & E \ar[d]^-{p}\\ A \ar@{.>}[ur]^-{f} \ar[r]^-{h} & B}
\end{equation}
in $\algosubc$ in which $p$ is a trivial fibration, i.e., an entrywise trivial fibration.  We must show that there exists a dotted arrow $f$ that makes the lower triangle commutative.  We first equip both $E$ and $B$ with the same color set $\fC$ with $h$ and $p$ the identity map on colors, so now $h,p \in \algso$.  

We next consider the diagram \eqref{algso.cof.lifting} as a diagram in $\algso$ in which $\varnothing$ is now the initial object in $\algso$.  Since the map $p$ is the identity map on colors and is a local trivial fibration, it is a trivial fibration in $\algso$.  By the cofibrancy assumption on $(\fC,A) \in \algso$, there exists a dotted arrow $f \in \algso$ such that $pf = h$.  Since $h$ and $p$ are both the identity map on colors, so is $f$.  Thus, $f$ gives a map $A \to E \in \algosubc$ and is the desired lift of $h$.
\end{proof}

\subsection{Nice Quillen Equivalences}
To lift a Quillen equivalence to the categories of algebras over operadic collections, we will need some extra cofibrancy conditions on the Quillen equivalence, which we now define.

\begin{definition}\label{def:sharp}
Suppose  $L : \M \adjoint \N : R$ is an adjunction between monoidal categories with $R$ a lax monoidal functor (Def. \ref{def:monoidal.functor}).  
\begin{enumerate}
\item For objects $X$ and $Y$ in $\M$, the map
\begin{equation}\label{comonoidal.map}
\nicexy{L\left(X \otimes Y\right) \ar[r]^-{L^2_{X,Y}} & LX \otimes LY \in \N,}
\end{equation}
defined as the adjoint of the composite
\[\nicexy{X \otimes Y \ar[r]^-{(\eta_X, \eta_Y)} & RLX \otimes RLY \ar[r]^-{R^2_{X,Y}} & R\bigl(LX \otimes LY\bigr),}\]
is called the \emph{comonoidal structure map} of $L$ \cite{ss03} (3.3).  Here $\eta$ is the unit of the adjunction.
\item Suppose further that $\M$ and $\N$ are symmetric monoidal categories that are also model categories with $R$ lax symmetric monoidal.  Recall the following condition from \cite{white-yau2} (2.4.3).
\begin{itemize}
\item[$(\#)$] Suppose $n \geq 1$, $W \in \Msigmanop$ is cofibrant in $\M$, and $X \in \Msigman$ is cofibrant in $\M$.  Then the map
\[\nicexy@C+.5cm{\bigl[L(W \otimes X)\bigr]_{\Sigma_n} \ar[r]^-{(L^2_{W,X})_{\Sigma_n}} & \bigl[LW \otimes LX\bigr]_{\Sigma_n}}\]
is a weak equivalence in $\N$, where $L^2_{W,X}$ is the comonoidal structure map of $L$ \eqref{comonoidal.map}.  
\end{itemize}
\end{enumerate}
\end{definition}

\begin{definition}\label{quillen.equivalence}
Suppose $L : \M \adjoint \N : R$ is a Quillen adjunction (Def. \ref{def:quillen.adjunction}).
\begin{enumerate}
\item We call $(L, R)$ a \emph{Quillen equivalence} \cite{hovey} (1.3.12) if for each cofibrant object $X$ in $\M$ and each fibrant object $Y$ in $\N$, a map $f : LX \to Y \in \N$ is a weak equivalence if and only if its adjoint $\fbar : X \to RY \in \M$ is a weak equivalence.  In this case, we call $L$ a \emph{left Quillen equivalence} and $R$ a \emph{right Quillen equivalence}.
\item Suppose $\M$ and $\N$ are monoidal model categories.  We call the Quillen adjunction $(L, R)$ a \emph{weak (symmetric) monoidal Quillen adjunction} \cite{ss03} (3.6) if the following three conditions hold.
\begin{enumerate}
\item $R$ is equipped with a lax (symmetric) monoidal structure.  
\item For some cofibrant replacement $q : Q\tensorunit^{\M} \to \tensorunit^{\M}$ of the monoidal unit in $\M$, the composite
\begin{equation}\label{unit.adjoint}
\nicexy{LQ\tensorunit^{\M} \ar[r]^-{Lq} & L\tensorunit^{\M} \ar[r]^-{\Rbar^0} & \tensorunit^{\N}}
\end{equation}
is a weak equivalence in $\N$, in which $\Rbar^0$ is the adjoint of the structure map $R^0 : \tensorunit^{\M} \to R\tensorunit^{\N}$.
\item For any cofibrant objects $X$ and $Y$ in $\M$, the comonoidal structure map \eqref{comonoidal.map} is a weak equivalence in $\N$.
\end{enumerate}
If, furthermore, $(L, R)$ is a Quillen equivalence, then we call it a \emph{weak (symmetric) monoidal Quillen equivalence}.
\end{enumerate}
\end{definition}

\begin{remark}
A weak monoidal Quillen adjunction subsumes a mild monoidal Quillen adjunction (Def. \ref{def:quillen.adjunction}).
\end{remark}

Next we consider some extra cofibrancy conditions on a model category with respect to a symmetric group action.  The condition $(\clubsuit)$ below was  first introduced in \cite{white-yau} (6.2.1), and $(\medstar)$ and $(\filledstar)$ were defined in \cite{white-yau2} (2.4.1).

\begin{definition}\label{def:extra.cofibrancy}
Suppose $\calm$ is a symmetric monoidal category and is a model category.  Define the following three conditions.
\begin{itemize}
\item[$(\clubsuit)$] For each $n \geq 1$ and $X \in \calm^{\sigmaop_n}$ that is cofibrant in $\calm$, the function
\[X \tensorover{\Sigma_n} (-)^{\boxprod n} : \calm \to \calm\]
preserves cofibrations and trivial cofibrations, where $(-)^{\boxprod n}$ is the $n$-fold iterate of the pushout product in Def. \ref{def:ppaxiom}.
\item[$(\filledstar)$] Suppose $n \geq 1$ and $X \in \Msigman$ is cofibrant in $\M$.  Then the coinvariant $X_{\Sigma_n} \in \M$ is also cofibrant.
\item[$(\medstar)$] Suppose $n \geq 1$, $g : U \to V \in \Msigmanop$ is a weak equivalence with $U$ and $V$ cofibrant in $\M$, and $X \in \Msigman$ is cofibrant in $\M$.  Then the map
\[\nicexy{U \tensoroversigman X \ar[r]^-{g \tensoroversigman X} & V \tensoroversigman X}\]
is a weak equivalence in $\M$.
\end{itemize}
\end{definition}

\begin{definition}\label{def:nice.qeq}
A \emph{nice Quillen equivalence} $L : \M \adjoint \N : R$ is a weak symmetric monoidal Quillen equivalence (Def. \ref{quillen.equivalence}) between monoidal model categories such that the following three conditions hold.
\begin{enumerate}
\item $(\#)$ (Def. \ref{def:sharp}), $(\clubsuit)$, and $(\filledstar)$ (Def. \ref{def:extra.cofibrancy}) hold in $\M$ and $\N$.  
\item $\N$ satisfies $(\medstar)$.
\item $\M$ is cofibrantly generated in which every generating cofibration has cofibrant domain.
\end{enumerate}
We call $\M$ \emph{nice} if $\Id : \M \adjoint \M : \Id$ is a nice Quillen equivalence, i.e., if $\M$ is a cofibrantly generated monoidal model category that satisfies $(\clubsuit)$, $(\filledstar)$, and $(\medstar)$, and if every generating cofibration in $\M$ has cofibrant domain.
\end{definition}

\begin{example}\label{ex:nice.quillen}
The following are examples of nice Quillen equivalences:
\begin{enumerate}
\item The Dold-Kan correspondence \eqref{dold-kan} between simplicial modules and non-negatively graded chain complexes of modules over a characteristic zero field.
\item The adjunction between reduced rational simplicial Lie algebras and reduced rational chain complexes of Lie algebras \cite{quillen} (p.211).
\end{enumerate}
Indeed, they are weak symmetric monoidal Quillen equivalences by \cite{ss03} (3.16 and 4.2).  The conditions $(\#)$, $(\clubsuit)$, $(\filledstar)$, and $(\medstar)$ hold if cofibrancy in $\Msigman$ coincides with cofibrancy in $\M$, which is true in the above characteristic zero cases.  In particular, these four monoidal model categories are nice and convenient (Def. \ref{def:convenient}).
\end{example}

\begin{definition}\label{def:sigma.cofibrant}
Suppose $\M$ is a symmetric monoidal closed category with a cofibrantly generated model category structure.   
\begin{enumerate}
\item A $\fC$-colored operad in $\M$ is said to be \emph{$\Sigma$-cofibrant} if its underlying colored symmetric sequence \eqref{symmetric.sequences} is cofibrant with respect to the diagram model category structure \cite{hirschhorn} (11.6.1), in which weak equivalences and fibrations are defined entrywise in $\M$.  
\item An operadic collection $(S,\sO)$ in $\M$ is said to be \emph{$\Sigma$-cofibrant} (resp., \emph{entrywise cofibrant}) if $\osubc \in \operadscm$ is $\Sigma$-cofibrant (resp., entrywise cofibrant) for each set $\fC$.
\end{enumerate}
\end{definition}

We will use the following result from \cite{white-yau2} (4.2.1 and 4.3.2), which is essentially the fixed color set version of the main result of this section.

\begin{theorem}\label{fixed.color.lifting}
Suppose:
\begin{itemize}
\item $(L,R)$ is a weak symmetric monoidal Quillen equivalence (Def. \ref{quillen.equivalence}) such that every generating cofibration in $\M$ has cofibrant domain.
\item $\alpha : \sO \to R\sP$ is a map of $\fC$-colored operads in $\M$ with $\fC$ a set, $\sO$  a $\fC$-colored operad in $\M$, and $\sP$ a $\fC$-colored operad in $\N$ such that the  entrywise adjoint $\alphabar : L\sO \to \sP$ is an entrywise weak equivalence in $\N$.  
\end{itemize}
Suppose further that one of the following two conditions holds:
\begin{enumerate}
\item $(L,R)$ is a nice Quillen equivalence (Def. \ref{def:nice.qeq}), and both $\sO$ and $\sP$ are entrywise cofibrant.
\item Both $\sO$ and $\sP$ are $\Sigma$-cofibrant.
\end{enumerate}
Then there is an induced Quillen equivalence
\begin{equation}\label{fixed.color.ralpha.adj}
\nicexy{\algo \ar@<2pt>[r]^-{\Lbar} & \algp \ar@<2pt>[l]^-{\Ralpha}}
\end{equation}
in which the right adjoint $\Ralpha$ is defined as in \eqref{Ralpha}.
\end{theorem}

\begin{remark}
In the above theorem in \cite{white-yau2} the categories $\algo$ and $\algp$ are semi-model categories \cite{fresse-book,fresse,hovey-monoidal,spitzweck-thesis} with entrywise defined weak equivalences and fibrations (Def. \ref{operad.alg.model}).  In our main result below our assumptions will be strong enough to ensure that the algebra categories under consideration are genuine model categories.  See Example \ref{ex:bm-admissible}.
\end{remark}

\subsection{Lifting Quillen Equivalences}
We are now ready for the main result of this section.

\begin{theorem}\label{lifting.qeq}
Under Assumption \ref{lifting.qadj.setting} suppose that:
\begin{itemize}
\item $(L,R)$ is a weak symmetric monoidal Quillen equivalence (Def. \ref{quillen.equivalence}) between convenient model categories (Def. \ref{def:convenient}) such that every generating cofibration in $\M$ has cofibrant domain.
\item For each set $\fC$, the entrywise adjoint
\[\nicexy{L\osubc \ar[r]^-{\alphacbar} & \psubc \in \N^{\profscsc}}\]
of the map $\alphac : \osubc \to R\psubc$ is an entrywise weak equivalence in $\N$.
\end{itemize}
Suppose further that one of the following two conditions holds:
\begin{enumerate}
\item $(L,R)$ is a nice Quillen equivalence (Def. \ref{def:nice.qeq}), and both operadic collections $(S,\sO)$ and $(S,\sP)$ are entrywise cofibrant.
\item Both $(S,\sO)$ and $(S,\sP)$ are $\Sigma$-cofibrant.
\end{enumerate}
Then the Quillen adjunction
\begin{equation}\label{lbar.ralpha.qeq}
\nicexy{\algso \ar@<2.5pt>[r]^-{\Lbar} & \algsp \ar@<2.5pt>[l]^-{\Ralpha}}
\end{equation}
in \eqref{lbar.ralpha.qadj} is a Quillen equivalence.
\end{theorem}

\begin{proof}
Suppose given a map
\[\nicexy{(\fC,A) \ar[r]^-{\varphi} & \Ralpha(\fD,B) = (\fD,RB) \in \algso}\]
such that $(\fC,A) \in \algso$ is cofibrant and $(\fD,B) \in \algsp$ is fibrant.  So $B$ is a locally fibrant $\psubd$-algebra.  The adjoint of the map $\varphi$ is a map
\begin{equation}\label{adjoint.varphi}
\nicexy{\Lbar(\fC,A) = \bigl(\fC,\Lbarc A\bigr) \ar[r]^-{\varphibar} & (\fD,B) \in \algsp,}
\end{equation}
which is the same as $\varphi$ on colors.  Here
\begin{equation}\label{ralphac.lbarc}
\nicexy{\algosubc \ar@<2pt>[r]^-{\Lbarc} & \algpsubc \ar@<2pt>[l]^-{\Ralphac}}
\end{equation}
is the Quillen equivalence in \eqref{fixed.color.ralpha.adj} induced by the map $\alphac : \osubc \to R\psubc$ of $\sofc$-colored operads.  The equality on the left side of \eqref{adjoint.varphi} comes from the fact that $\Ralphac$ is defined in the same way as $\Ralpha$ \eqref{Ralpha}.  We must show that $\varphi \in \algso$ is a weak equivalence if and only if $\varphibar \in \algsp$ is a weak equivalence.  

Consider the factorizations
\[\nicexy{(\fC,A) \ar[r]^-{\alpha} \ar `u[rr] `[rr]|-{\varphi} [rr] & (\fC,\varphi^* RB) \ar[r]^-{\beta} & (\fD,RB)} \in \algso\]
and
\[\nicexy{\bigl(\fC,\Lbarc A\bigr) \ar[r]^-{\alphabar} \ar `u[rr] `[rr]|-{\varphibar} [rr] & (\fC,\varphibar^*B) \ar[r]^-{\betabar} & (\fD,B)} \in \algsp\]
as in \eqref{f.factorization}.  Using the abbreviation and essentially the same calculation as in \eqref{pizero.rb}, there is an identification
\begin{equation}\label{pizero.identification}
\pizero(\beta^{\cat}) \cong \pizero(\betabar^{\cat}) \in \Cat.
\end{equation}
By Def. \ref{def:algso.weq.fib} of the Dwyer-Kan model category structure on $\algso$, the map $\varphi$ is a weak equivalence if and only if
\begin{enumerate}
\item $\alpha \in \algosubc$ is a weak equivalence in the projective model category structure (Def. \ref{operad.alg.model}), and
\item $\pizero(\beta^{\cat}) \in \Cat$ is essentially surjective.
\end{enumerate}
A similar remark applies to $\varphibar$.  By \eqref{pizero.identification} the map $\pizero(\beta^{\cat}) \in \Cat$ is essentially surjective if and only if $\pizero(\betabar^{\cat}) \in \Cat$ is so.  It remains to show that $\alpha$ is a weak equivalence if and only if $\alphabar$ is a weak equivalence.

Since $B$ is locally fibrant, the $\psubc$-algebra $\varphibar^*B$, whose entries are among those of $B$, is fibrant in the projective model category structure.  Moreover, we have an equality
\[\varphi^*RB = \Ralphac (\varphibar^\ast B) \in \algosubc.\]
Since $(\fC,A) \in \algso$ is cofibrant in the Dwyer-Kan model category structure,  $A \in \algosubc$ is cofibrant in the projective model category structure by Lemma \ref{cofibrant.algebra}.  Using the Quillen equivalence $(\Lbarc, \Ralphac)$ in \eqref{ralphac.lbarc}, we infer that the map
\[\nicexy{A \ar[r]^-{\alpha} & \varphi^*RB = \Ralphac (\varphibar^\ast B) \in \algosubc}\]
is a weak equivalence if and only if its adjoint
\[\nicexy{\Lbarc A \ar[r]^-{\alphabar} & \varphibar^\ast B \in \algpsubc}\]
is a weak equivalence.  We have shown that $\varphi$ is a weak equivalence if and only if $\varphibar$ is a weak equivalence.
\end{proof}

\subsection{Application to Operad-Like Structure}\label{operad.like.lifting}

Suppose $L : \M \adjoint \N : R$ is as in Theorem \ref{lifting.qeq}(1), i.e., a nice Quillen equivalence (Def. \ref{def:nice.qeq}) between convenient model categories (Def. \ref{def:convenient}), such as those in Example \ref{ex:nice.quillen}.  Suppose $(S,\Gm)$ is an operadic collection in $\M$ for a chosen operad-like structure in Example \ref{ex:gprops}, and similarly for $(S,\Gn)$ in $\N$ for the same operad-like structure.  In this setting, we will observe that Theorem \ref{lifting.qeq} applies.  First, both $(S,\Gm)$ and $(S,\Gn)$ are augmented operadic collections by Example \ref{ex:operad-like-augmented}.

For each set $\fC$, each entry of the $\sofc$-colored operad $\Gmc$ has the form $\coprod \tensorunitm$, and the corresponding entry of $\Gnc$ has the form $\coprod \tensorunitn$.  The two coproducts are indexed by the same set of isomorphism classes of graphs.  Since the monoidal units are cofibrant, the coproducts $\coprod \tensorunitm$ and $\coprod \tensorunitn$ are cofibrant.  So both operadic collections $(S,\Gm)$ and $(S,\Gn)$ are entrywise cofibrant.  

Define a map
\[\alpha : (S,\Gm) \to (S,R\Gn)\]
of operadic collections in $\M$ as follows.  For each set $\fC$, each entry of the map
\[\nicexy{\Gmc \ar[r]^-{\alphac} & R\Gnc \in \operadscm}\]
is the adjoint of the coproduct of maps
\[\nicexy{(L\Gmc)(\vdots) = L\bigl(\coprod \tensorunitm\bigr)  \cong \coprod L\tensorunitm \ar[r]^-{\alphabarc} & \coprod \tensorunitn = \Gnc(\vdots)}.\]
By our assumptions on $(L,R)$, the map $L\tensorunitm \to \tensorunitn$, which is adjoint to the monoidal functor structure map $\tensorunitm \to R\tensorunitn$, is a weak equivalence between cofibrant objects.  It follows that the entry map $\alphabarc$ above is also a weak equivalence between cofibrant objects.  Moreover, the augmentation compatibility diagram \eqref{augmentation.compatible} is commutative here.  In fact, the map called $\epsilon$ there is entrywise the adjoint of a sub-coproduct of the map $\alphabarc$ above, corresponding to the subset of isomorphism classes of linear graphs.

Therefore, Theorem \ref{lifting.qeq}(1) applies in this setting, and we obtain the following induced Quillen equivalences:
\[\begin{split}
\bigl(\text{vprops in $\M$}\bigr) &\adjoint \bigl(\text{vprops in $\N$}\bigr)\\
\bigl(\text{Colored operads in $\M$}\bigr) &\adjoint \bigl(\text{Colored operads in $\N$}\bigr)\\
\bigl(\text{Dioperads in $\M$}\bigr) &\adjoint \bigl(\text{Dioperads in $\N$}\bigr)\\
\bigl(\text{Properads in $\M$}\bigr) &\adjoint \bigl(\text{Properads in $\N$}\bigr)\\
\bigl(\text{Props in $\M$}\bigr) &\adjoint \bigl(\text{Props in $\N$}\bigr)\\
\bigl(\text{Wheeled operads in $\M$}\bigr) &\adjoint \bigl(\text{Wheeled operads in $\N$}\bigr)\\
\bigl(\text{Wheeled properads in $\M$}\bigr) &\adjoint \bigl(\text{Wheeled properads in $\N$}\bigr)\\
\bigl(\text{Wheeled props in $\M$}\bigr) &\adjoint \bigl(\text{Wheeled props in $\N$}\bigr)
\end{split}\]
Under weaker assumptions, a similar Quillen equivalence between categories of small enriched categories is \cite{muro15} (1.4).


\end{document}